\def\draft{0}
\def\anon{0}

\documentclass[12pt]{article}
\usepackage[utf8]{inputenc}
\usepackage{fullpage}
\usepackage{amsmath}
\usepackage{amsfonts}
\usepackage{amssymb}
\usepackage{graphicx}
\usepackage{url}
\usepackage[dvipsnames]{xcolor}

\usepackage{microtype}      
\usepackage{libertine}
\usepackage{libertinust1math}

\DeclareMathOperator{\Span}{span}

\DeclareMathOperator{\Norm}{\mathcal{N}}

\definecolor{DarkBlue}{RGB}{0,0,150}
\definecolor{NotSoDarkBlue}{RGB}{15,15,210}
\definecolor{DarkRed}{RGB}{150,0,0}
\definecolor{DarkGreen}{RGB}{0,100,0}
\usepackage[pdfstartview=FitH,pdfpagemode=UseNone,colorlinks,linkcolor=DarkRed,filecolor=blue,citecolor=DarkRed,urlcolor=DarkRed,pagebackref,breaklinks]{hyperref}
\usepackage{cleveref}

\newcommand{\R}{\mathbb{R}}
\newcommand{\N}{\mathbb{N}}
\newcommand{\Z}{\mathbb{Z}}

\newcommand{\one}{\mathbf{1}}
\newcommand{\zero}{\mathbf{0}}

\newcommand{\oracle}{\mathcal{O}}

\newcommand{\algo}[1]{\ensuremath{\mathsf{#1}}}

\newcommand{\negl}{\algo{negl}}

\newcommand{\poly}{\algo{poly}}
\newcommand{\polylog}{\algo{polylog}}

\newcommand{\PP}{\mathcal{P}}

\newcommand{\LL}{\mathcal{L}}

\usepackage{braket}
\usepackage{amsthm}
\newtheorem{theorem}{Theorem}
\newtheorem{lemma}{Lemma}

\newtheorem{remark}{Remark}
\newtheorem{assumption}{Assumption}
\newtheorem{proposition}{Proposition}
\newtheorem{corollary}{Corollary}

\newtheorem{definition}{Definition}

\def\veca{\mathbf{a}}
\def\vecb{\mathbf{b}}
\def\vecc{\mathbf{c}}
\def\vece{\mathbf{e}}
\def\vecf{\mathbf{f}}
\def\vecm{\mathbf{m}}
\def\vecp{\mathbf{p}}
\def\vecs{\mathbf{s}}
\def\vecx{\mathbf{x}}
\def\vecu{\mathbf{u}}
\def\vecv{\mathbf{v}}
\def\vecw{\mathbf{w}}
\def\vecy{\mathbf{y}}
\def\vecz{\mathbf{z}}
\def\vect{\mathbf{t}}
\def\vecmu{\boldsymbol{\mu}}

\def\matA{\mathbf{A}}
\def\matB{\mathbf{B}}
\def\matE{\mathbf{E}}
\def\matI{\mathbf{I}}
\def\matM{\mathbf{M}}
\def\matS{\mathbf{S}}
\def\matU{\mathbf{U}}
\def\matV{\mathbf{V}}
\def\matW{\mathbf{W}}

\def\matSigma{\mathbf{\Sigma}}
\newcommand{\norm}[1]{\left\| {#1} \right\|}
\newcommand{\inner}[1]{\left\langle {#1} \right\rangle}
\newcommand{\floor}[1]{\left\lfloor {#1} \right\rfloor}
\newcommand{\ceil}[1]{\left\lceil {#1} \right\rceil}

\newcommand{\SBP}{\mathrm{SBP}}
\newcommand{\NPP}{\mathrm{NPP}}
\newcommand{\ABP}{\mathrm{ABP}}
\newcommand{\SIVP}{\mathrm{SIVP}}
\newcommand{\GDD}{\mathrm{GDD}}
\newcommand{\IncGDD}{\mathrm{IncGDD}}
\newcommand{\GapCRP}{\mathrm{GapCRP}}

\newcommand{\sbpError}{\kappa}
\newcommand{\nppError}{\kappa}

\newcommand{\iid}{\mathrm{i.i.d.}}

\usepackage{soul}

\ifnum\draft=1
\def\ShowAuthNotes{1}
\else
\def\ShowAuthNotes{0}
\fi

\ifnum\ShowAuthNotes=1
\newcommand{\authnote}[3]{\textcolor{#3}{[{\footnotesize {\bf #1:} { {#2}}}]}}
\else
\newcommand{\authnote}[3]{}
\fi

\title{{Symmetric Perceptrons, Number Partitioning \\ and Lattices}}
\ifnum\anon=0
\author{Neekon Vafa\\MIT\\\href{mailto:nvafa@mit.edu}{nvafa@mit.edu} \and Vinod Vaikuntanathan\\MIT\\\href{mailto:vinodv@mit.edu}{vinodv@mit.edu}}
\else
\author{Anonymous authors}
\fi
\date{}
\begin{document}
\maketitle

\begin{abstract}
    \noindent
The \emph{symmetric binary perceptron} ($\mathrm{SBP}_{\kappa}$) problem with parameter $\kappa : \mathbb{R}_{\geq1} \to [0,1]$ is an average-case search problem defined as follows: given a random Gaussian matrix $\mathbf{A} \sim \mathcal{N}(0,1)^{n \times m}$ as input where $m \geq n$, output a vector $\mathbf{x} \in \{-1,1\}^m$ such that $$|| \mathbf{A} \mathbf{x} ||_{\infty} \leq \kappa(m/n) \cdot \sqrt{m}~.$$
The \emph{number partitioning problem} ($\mathrm{NPP}_{\kappa}$) corresponds to the special case of setting $n=1$. There is considerable evidence that both problems exhibit large computational-statistical gaps.
    
In this work, we show (nearly) tight \emph{average-case} hardness for these problems, assuming the \emph{worst-case} hardness of standard approximate shortest vector problems on lattices.

\begin{itemize}
\item For $\mathrm{SBP}_\kappa$, statistically, solutions exist with $\kappa(x) = 2^{-\Theta(x)}$ (Aubin, Perkins and Zdeborov\'{a}, Journal of Physics 2019). For large $n$, the best that efficient algorithms have been able to achieve is a far cry from the statistical bound, namely $\kappa(x) = \Theta(1/\sqrt{x})$ (Bansal and Spencer, Random Structures and Algorithms 2020). The problem has been extensively studied in the TCS and statistics communities, and Gamarnik, K{\i}z{\i}lda\u{g}, Perkins and Xu (FOCS 2022) conjecture that Bansal-Spencer is tight: namely, $\kappa(x) = \widetilde{\Theta}(1/\sqrt{x})$ is the optimal value achieved by computationally efficient algorithms. 
        
We prove their conjecture assuming the worst-case hardness of approximating the shortest vector problem on lattices.

\item For $\mathrm{NPP}_\kappa$, statistically, solutions exist with $\kappa(m) = \Theta(2^{-m})$ (Karmarkar, Karp, Lueker and Odlyzko, Journal of Applied Probability 1986). Karmarkar and Karp's classical differencing algorithm achieves $\kappa(m) = 2^{-O(\log^2 m)}~.$
        
We prove that Karmarkar-Karp is nearly tight: namely, no polynomial-time algorithm can achieve $\kappa(m) = 2^{-\Omega(\log^3 m)}$, once again assuming the worst-case subexponential hardness of approximating the shortest vector problem on lattices to within a subexponential factor.
\end{itemize}

\noindent
Our hardness results are versatile, and hold with respect to different distributions of the matrix $\mathbf{A}$ (e.g., i.i.d. uniform entries from $[0,1]$) and weaker requirements on the solution vector $\mathbf{x}$. 
\end{abstract}
\newpage 

\section{Introduction}

\paragraph{Symmetric Binary Perceptrons.} 
The \emph{symmetric binary perceptron} ($\SBP_{\sbpError}$) problem, also called the symmetric Ising perceptron problem~\cite{JH60,Win61,Cover,Aubin_2019,BDVLZ20,DBLP:conf/stoc/PerkinsX21,abbe2021proofcontiguityconjecturelognormal,DBLP:conf/stoc/AbbeLS22,DBLP:conf/focs/GamarnikK0X22,DBLP:conf/colt/GamarnikK0X23,Barbier_2024}, is a search problem defined as follows: given a random matrix $\matA \sim \mathcal{N}(0,1)^{n \times m}$ with entries chosen i.i.d. from the normal distribution where $m \geq n$, find a binary vector $\vecx \in \{-1,1\}^m$ such that 
$\norm{\matA \vecx}_{\infty}$ is minimized. 
More formally, SBP with 
parameter $\sbpError : \R_{\geq 1} \to [0,1]$ asks us to find an $\vecx \in \{-1,1\}^m$ such that 
$$\norm{\matA \vecx}_{\infty} \leq \sbpError(m/n) \cdot \sqrt{m}~.$$
Here, the quality of the solution is parameterized as a function of $m/n$,  the so-called (inverse) aspect ratio of the problem.\footnote{The notation commonly used in the literature to parameterize $\SBP$ (and $\NPP$) is slightly different than the notation we choose to use. The aspect ratio is given by $\alpha = n/m$, and instead of writing $\sbpError$ as a function of $\alpha$ (or really, $x = 1/\alpha$, as we do), the roles are flipped, where $\alpha$ is a function of $\sbpError$. For example, $\sbpError(x) = 2^{-x}$ and $\sbpError(x) = 1/\sqrt{x}$, in our notation, correspond to $\alpha(\sbpError) = 1/\log_2(1/\kappa)$ and $\alpha(\sbpError) = \sbpError^2$ in the notation of \cite{DBLP:conf/focs/GamarnikK0X22}, respectively.}

The problem is versatile, and can be defined with respect to different distributions of the matrix $\matA$ (i.i.d. uniform $[0,1]$ entries is another popular choice) and different requirements on the solution vector $\vecx$. The problem also has rich connections to several other fields including the classical subset sum problem and its variants and the problem of  discrepancy minimization.

Two natural questions arise: a statistical question and a computational one. The statistical question asks for which $\sbpError$ do solutions exist (with high probability over the choice of the matrix $\matA$).  Recent works by Aubin, Perkins and Zdeborov\'{a}~\cite{Aubin_2019}, Perkins and Xu~\cite{DBLP:conf/stoc/PerkinsX21}, and Abbe, Li and Sly~\cite{abbe2021proofcontiguityconjecturelognormal} showed sharp statistical thresholds for this problem. In particular, they showed that the threshold for the existence of solutions is 
$$\sbpError_{\mathsf{stat}}(x) = O(2^{-x}).$$
On the other hand, the best solutions found by polynomial-time algorithms satisfy 
\[ \sbpError_{\mathsf{comp}}(x) = \Omega\left(\frac{1}{\sqrt{x}}\right).\footnote{In an extreme parameter regime where $n = O\left(\sqrt{\log m}\right)$, \cite{DBLP:conf/colt/TurnerMR20} gives a polynomial-time algorithm achieving discrepancy $2^{-\Omega(\log^2(m)/n)}$, but throughout our paper, for $\SBP$, we will only consider the regime where $n = \omega(\log m)$.}\]
This comes from the breakthrough work of Bansal~\cite{DBLP:conf/focs/Bansal10} and Bansal and Spencer~\cite{BS20} from the closely related field of combinatorial discrepancy theory.\footnote{Their result is established for the case of matrices $\matA$ with  i.i.d. Rademacher entries. Nevertheless, \cite{DBLP:conf/focs/GamarnikK0X22} conjecture that the same guarantee remains true for the case of i.i.d. standard normal entries.} Gamarnik, K{\i}z{\i}lda\u{g}, Perkins and Xu~\cite{DBLP:conf/focs/GamarnikK0X22,DBLP:conf/colt/GamarnikK0X23} studied the large statistical-computational gap scenario in detail and conjectured that no polynomial-time algorithms can achieve a guarantee much better than $\sbpError_{\mathsf{comp}}(x)$. In particular, they show that the so-called overlap gap property~\cite{Mezard,DBLP:conf/stoc/AchlioptasR06,GL18}, which rules out a class of stable algorithms, sets in at $$\sbpError_{\mathsf{overlap}}(x) = O\left(\frac{1}{\sqrt{x\log x}}\right).$$
We refer the reader to \cite{DBLP:conf/focs/GamarnikK0X22} for an extensive discussion of the rationale behind their conjecture.

A survey of Gamarnik on the overlap gap property~\cite{gamarnik2021overlap} points to the question of whether average-case hardness of perceptron problems (and more) can be based on \emph{worst-case} hardness assumptions. Gamarnik explicitly states that worst-case to average-case reductions for these problems ``would be ideal for our setting, as they would provide the most compelling
evidence of hardness of these problems'' \cite{gamarnik2021overlap}.

The first contribution of this work is a proof of the conjecture of \cite{DBLP:conf/focs/GamarnikK0X22} upto lower order terms, under the assumption that standard, well-studied, lattice problems are hard to approximate in the worst case.

\begin{theorem}[Informal version of \Cref{sivp-to-sbp}]\label{main-sbp-intro}
Let $\varepsilon > 0$ be any constant.
Assuming that ${\gamma}$-approximate lattice problems in $n$ dimensions with $\gamma(n) = n^{O(1/\varepsilon)}$ are worst-case hard for polynomial-time algorithms, $\SBP_{\sbpError}$ with 
$$\sbpError(x) = \frac{1}{x^{1/2 + \varepsilon}}$$ 
is hard for polynomial-time algorithms. 
Additionally, assuming near-optimal worst-case hardness of lattice problems, we obtain near-optimal hardness of SBP. In particular, assuming that $\gamma(n)$-approximate lattice problems require $2^{\omega(n^{1/2 - \varepsilon})}$ time to solve in the worst case with $\gamma(n) = 2^{n^{1/2 - \varepsilon}}$, we have that $\SBP_{\sbpError}$ with 
$$\sbpError(x) = \frac{1}{\sqrt{x}\cdot (\log x)^{c}}$$
is hard for some constant $c > 1$.
\end{theorem}

Lattice problems such as the shortest vector problem and the shortest independent vectors problem have been extensively studied for decades largely for their implications to combinatorial optimization~\cite{lenstra1982factoring,DBLP:conf/stoc/Kannan83,DBLP:journals/mor/Kannan87,DBLP:conf/focs/ReisR23} and even more so, to cryptography~\cite{DBLP:conf/stoc/Ajtai96,DBLP:journals/siamcomp/MicciancioR07,DBLP:journals/jacm/Regev09}. 
Time-approximation tradeoffs for lattice problems are well-known~\cite{DBLP:journals/tcs/Schnorr87}.
The best known algorithms for these lattice problems employ \emph{lattice reduction} techniques, based on the LLL and BKZ algorithms~\cite{lenstra1982factoring,DBLP:journals/tcs/Schnorr87}. They currently all have the following time-approximation trade-off: For a tunable parameter $k$, in dimension $n$, it is possible to solve lattice problems with approximation factor $\gamma = 2^{\widetilde{O}(n/k)}$ in time $2^{\widetilde{O}(k)}$, where $\widetilde{O}(\cdot)$ hides polylog factors in $n$. Equalizing these terms gives a $2^{\widetilde{O}(\sqrt{n})}$-time algorithm to solve $2^{\widetilde{O}(\sqrt{n})}$-approximate lattice problems. Despite extensive study in the lattice literature, no better algorithms are known (that improve the exponent by more than a polylog factor). This motivates the assumptions in our theorem statement, both the conservative one and the near-optimal one. (For more discussion on this, see Section~\ref{latticeprelims}).

Our reduction has several additional features: first, our reduction is versatile and works with respect to different distributions of the matrix $\matA$ (e.g., i.i.d. uniform entries from $[0,1]$); and secondly, it shows the hardness not just of computing an SBP solution with $1-o(1)$ probability, but indeed with any non-trivial inverse polynomial probability. 
For more implications of our reduction, we refer the reader to 
Section~\ref{sec:variants-and-generalizations}. Moreover, our full reduction is conceptually simple and direct. We discuss the possibility of using other intermediate problems to establish these same results in \Cref{sec:alternate-reductions}.

\paragraph{Number Partitioning (or Number Balancing).} The (average-case) {\em number partitioning} problem is a special case of SBP and corresponds to setting $n=1$ in SBP. Given $m$ random numbers $a_1,\ldots,a_m \sim \mathcal{N}(0,1)$ with entries chosen i.i.d. from the standard normal distribution, the goal is to find a binary vector $\vecx \in \{-1,1\}^m$ such that 
$|\sum_{i=1}^m x_i a_i |$ is minimized. 
More formally, $\NPP$ with 
parameter $\sbpError : \Z_{\geq 1} \to [0,1]$ asks us to find an $\vecx \in \{-1,1\}^m$ such that 
$$\left|\sum_{i=1}^m x_i a_i \right|  \leq \nppError(m) \cdot\sqrt{m} $$

Number partitioning, as a worst-case problem, was one of the six original NP-complete problems in the classic book of Garey and Johnson~\cite{DBLP:books/fm/GareyJ79}. The worst-case version of the problem, where $a_1,\ldots,a_m$ are arbitrary, is closely related to discrepancy problems and has been extensively studied; see \cite{spencer,DBLP:conf/focs/Bansal10,DBLP:journals/siamcomp/LovettM15,DBLP:conf/ipco/LevyRR17,DBLP:conf/focs/Rothvoss14,DBLP:conf/ipco/HobergRRY17}. In particular, \cite{DBLP:conf/ipco/HobergRRY17} show that there are (worst-case to worst-case) reductions between $\NPP$ and worst-case lattice problems.\footnote{Technically, the problem they consider is the number balancing problem, where they require the weaker condition on the solution $\vecx$ that $\vecx \in \{-1, 0, 1\}^m \setminus \{\zero\}$ instead of $\vecx \in \{-1, 1\}^m$. As we explain later (\Cref{sec:variants-and-generalizations}), our reduction works in this setting as well.} Their reduction from worst-case lattice problems to $\NPP$ \cite[Theorem 9]{DBLP:conf/ipco/HobergRRY17} shows hardness for $\nppError(m) \leq 2^{-m/2}$. Looking ahead, we show computational hardness for the average-case version of $\NPP$ and for the much tighter range of $\nppError(m) = 2^{- \polylog(m)}$.

An application of the pigeonhole principle shows that solutions exist, both in the worst case and on average (even with high probability \cite{karmarkar1986probabilistic}), for $$\kappa_{\mathsf{stat}}(m) = 2^{-m}.$$ 
However, the best solutions found by polynomial-time algorithms satisfy 
$$ \sbpError_{\mathsf{comp}}(m) = 2^{-O(\log^2 m)}.$$
This comes from the beautiful ``differencing algorithm'' of Karmarkar and Karp~\cite{KK82} which starts with a list; sorts it; replaces the largest and second largest element with their absolute difference; and repeats until a single element is left which the algorithm outputs as the discrepancy of the set of numbers. (An informed reader might already have observed the analogy of Karmarkar-Karp with the Blum-Kalai-Wasserman~\cite{DBLP:journals/jacm/BlumKW03} algorithm, which came much later.) A subsequent work of Yakir~\cite{Yakir} proved that their algorithm indeed achieves the claimed discrepancy of $ \sbpError_{\mathsf{comp}}(m) = 2^{-O(\log^2 m)}.$ This remains the best algorithm known to date. 

Gamarnik and K{\i}z{\i}lda\u{g}~\cite{gamarnik2021algorithmicobstructionsrandomnumber} studied the statistical-computational gap in depth, demonstrated an overlap gap property at $$\kappa^*(m) = 2^{-\omega\left(\sqrt{m\log m}\right)},$$ implying that the class of stable algorithms will fail to find solutions for such small $\kappa$. This leaves open the question of the ground truth: are there improvements to Karmarkar-Karp, stable or not, that efficiently solve $\NPP_{\nppError}$ with $\nppError(m) = 2^{-m^{\Omega(1)}}$?

Our second contribution is proving that the answer to this question is \emph{no}, in a quantitatively strong way. In particular, under the assumption that standard, well-studied, lattice problems are sub-exponentially hard to approximate in the worst case, we prove that Karmarkar-Karp is tight, up to a logarithmic factor in the exponent.

\begin{theorem}[Informal version of \Cref{sivp-to-npp}]\label{informal-npp-result}
Suppose $\nppError(m) = 2^{-\log^{3 + \varepsilon} m}$ for some constant $\varepsilon > 0$. Assuming near-optimal hardness of worst-case lattice problems, then $\NPP_{\nppError}$ is hard for polynomial-time algorithms. In particular, assuming that $\gamma(n)$-approximate lattice problems in dimension $n$ require $2^{\omega(n^{1/2 - \varepsilon})}$ time to solve in the worst case with $\gamma(n) = 2^{n^{1/2 - \varepsilon}}$, we have that $\NPP_{\nppError}$ (in dimension $m$) is hard for $\poly(m)$-time algorithms.
\end{theorem}

Similar to the case of our $\SBP$ result, this theorem is quite versatile. We refer the reader to the technical overview and \Cref{sec:nbp} for more details. Like in the case $\SBP$, our full reduction is conceptually simple and direct, and we discuss the possibility of using other intermediate problems to establish these same results in \Cref{sec:alternate-reductions}.

\paragraph{Adaptive Robustness of Johnson-Lindenstrauss.} The Johnson-Lindenstrauss lemma \cite{johnson1984extensions} states that for all fixed, small, finite sets $S \subseteq \R^m$, for random $\matA \sim \Norm(0,1)^{n \times m}$, the linear map given by $\matA$ embeds $S$ into $\R^n$ in a way that approximately preserves all $\ell_2$ norms (up to a $\sqrt{n}$ normalization term). Slightly more concretely, 
\[ \forall \text{ small, finite } S \subseteq \R^m,\;\;\;\Pr_{\matA \sim \Norm(0, 1)^{n \times m}}\left[ \forall \vecx \in S,\;\norm{\matA \vecx}_2 \approx \sqrt{n} \norm{\vecx}_2 \right] = 1 - o(1). \]
This statement crucially relies on the fact the set $S$ that is defined independently of $\matA$. For example, even considering only singleton sets $S$, one can ask whether the order of quantifiers can be switched so that $\vecx$ can be chosen adaptively based on $\matA$, in the following sense:
\[ \Pr_{\matA \sim \Norm(0, 1)^{n \times m}}\left[ \forall \vecx \in \R^m,\;\norm{\matA \vecx}_2 \approx \sqrt{n} \norm{\vecx}_2 \right] \stackrel{?}{=} 1 - o(1). \]
Or even weaker, for a function $\sbpError : \R_{>1} \to (0,1]$,
\[ \Pr_{\matA \sim \Norm(0, 1)^{n \times m}}\left[ \forall \vecx \in \R^m,\;\norm{\matA \vecx}_2 \geq \norm{\vecx}_2  \sbpError(m/n) \sqrt{n}\right] \stackrel{?}{=} 1 - o(1). \]
However, this is impossible. For $m > n$, one can find a vector $\vecx \in \ker(\matA)$ (so $S = \{\vecx\}$), making $\norm{\matA \vecx}_2 = 0$ while $\norm{\vecx}_2$ can be arbitrarily large.

A natural question is whether this phenomenon could hold if we constrain $\vecx \in \R^m$ to some structured set, e.g., $\{-1, 1\}^m$:
\[ \Pr_{\matA \sim \Norm(0, 1)^{n \times m}}\left[ \forall \vecx \in \{-1, 1\}^m,\;\norm{\matA \vecx}_2 \geq \sbpError(m/n) \sqrt{nm}\right] \stackrel{?}{=} 1 - o(1). \]
This question is exactly the statistical capacity of $\SBP_{\kappa}$, with the exception that the norm on $\matA \vecx$ has changed from $\ell_{\infty}$ to $\ell_2$ (which differ by only a $\sqrt{n}$ factor at most).

Therefore, we view $\SBP_{\kappa}$ as defining a natural adaptive robustness question about a certain discretized version of the Johnson-Lindenstrauss lemma. In particular, since $\SBP_{\kappa}$ exhibits a computational-statistical gap, so does this variant of the Johnson-Lindenstrauss lemma. Given the utility of the Johnson-Lindenstrauss lemma in compressed sensing, dimensionality reduction, and more, we believe that this interpretation may inspire connections between worst-case lattice problems and adaptive robustness of downstream applications of the Johnson-Lindenstrauss lemma. 

\paragraph{Open Questions and Future Directions.} A direct open question raised by our results is to better understand the gap between \Cref{informal-npp-result} and Karmarkar-Karp~\cite{KK82}. Specifically, our result shows hardness for $\nppError(m) = 2^{-\log^{3 + \varepsilon} m}$, but Karmarkar-Karp gives a polynomial time algorithm that achieves $\nppError(m) = 2^{-\Theta(\log^2 m)}$. Can Karmarkar-Karp be improved to $\nppError(m) = 2^{-\Theta(\log^3 m)}$, can the hardness shown in \Cref{informal-npp-result} be improved, or is the truth somewhere in the middle? 

For $\SBP$, one can ask what happens in the setting where $m = \Theta(n)$, i.e., so the ratio $m/n$ is $\Theta(1)$. The proof of our hardness result crucially needs $m \geq n^{\Theta(1/\epsilon)}$, with no setting of our parameters yielding $m = \Theta(n)$. Can one extend hardness like in \Cref{main-sbp-intro} to this proportional setting?

Another question one can ask is related to the \emph{asymmetric} binary perceptron ($\ABP_{\sbpError}$) problem, also called the asymmetric Ising perceptron problem, which for $\matA \sim \Norm(0, 1)^{n \times m}$, asks to find $\vecx \in \{-1, 1\}^m$ such that
\[ \matA \vecx \geq \sbpError(m/n) \sqrt{m} \cdot \one, \]
in the sense that for every row $\veca_j \in \R^m$ for $j \in [n]$, we want $\veca_j^\top \vecx \geq \sbpError(m/n) \sqrt{m}$. (For more details, see e.g.,  \cite{DBLP:conf/focs/GamarnikK0X22}.) Note that unlike $\SBP_{\sbpError}$, setting $\sbpError(x) = 0$ still defines a meaningful problem. Does $\ABP$ share similar hardness results (from worst-case problems)? If so, are lattice problems the source of such hardness?

More generally, we can ask the \emph{converse} questions of the ones raised in our paper. Can lattice algorithms be used to improve algorithms for $\NPP$, $\SBP$ or $\ABP$? Recent work has given some hope along these lines, showing that lattice-based methods outperform sum-of-squares~\cite{DBLP:conf/colt/ZadikSWB22}.

We leave all of these as fascinating open questions and future directions of our work.

\paragraph{Related work.} Recently, various optimization and learning problems have been shown to have lattice-based hardness. The formative work of Bruna, Regev, Song and Tang~\cite{DBLP:conf/stoc/Bruna0ST21} defined a variant of the standard cryptographic Learning With Errors (LWE) problem called \emph{Continuous Learning With Errors} (CLWE), which is an average-case, periodic, linear regression task, which they showed is as hard as worst-case lattice problems. This and subsequent works have used CLWE as an intermediate problem to show worst-case lattice hardness of learning mixtures of Gaussians, learning single periodic neurons, learning halfspaces, detecting planted backdoors in certain machine learning models, and more \cite{DBLP:conf/stoc/Bruna0ST21,DBLP:conf/nips/SongZB21,DBLP:conf/focs/GupteVV22,DBLP:conf/nips/DiakonikolasKMR22,DBLP:conf/icml/DiakonikolasKR23,DBLP:conf/colt/Tiegel23,DBLP:conf/focs/GoldwasserKVZ22}.

\subsection{Technical Overview}
For both of our reductions, we take direct inspiration from worst-case to average-case reductions in the lattice literature~\cite{DBLP:conf/stoc/Ajtai96,DBLP:journals/siamcomp/MicciancioR07}. Our exposition and proofs closely follow \cite{DBLP:journals/siamcomp/MicciancioR07}. We view our work as bringing techniques from worst-case to average-case reductions in the lattice literature to closely related problems in statistics and discrete optimization. 

\subsubsection{Reduction to \texorpdfstring{$\SBP$}{SBP}}

We outline the main ideas behind the reduction from the worst-case lattice problems to $\SBP$. For an invertible matrix $\matB \in \R^{n \times n}$, let $\LL(\matB)$ denote the lattice generated by (the columns of) $\matB$, i.e.,
\[ \LL(\matB) = \{\matB \vecz : \vecz \in \Z^n\}.\]
Let $\PP(\matB)$ denote the \emph{fundamental parallelepiped} of $\matB$, given by the set $\matB [0,1)^n$.

To illustrate the ideas in our reduction, consider the following (informal) worst-case lattice problem: given some invertible lattice basis $\matB \in \R^{n \times n}$, output a (non-zero) vector $\vecs \in \LL(\matB)$ slightly smaller than the columns in $\matB$. (While this is not quite the worst-case lattice problem we reduce from, it is simpler to describe this way and captures all of the main ideas. See \Cref{def:incgdd} for more precise details.)

The key idea we employ is \emph{smoothing} using Gaussian measures, as introduced by Micciancio and Regev~\cite{DBLP:journals/siamcomp/MicciancioR07}. Specifically, here is the critical point: For an \emph{arbitrary} invertible $\matB \in \R^{n \times n}$ and for $\vecu \sim \Norm(\zero, \sigma^2 \matI_n)$, as long as $\sigma$ is large enough, $\vecu \pmod{\LL(\matB)}$ looks \emph{statistically} close to uniform over $\PP(\matB)$. This holds as long as $\sigma$ is larger than the \emph{smoothing parameter} of the lattice $\LL(\matB)$, defined by \cite{DBLP:journals/siamcomp/MicciancioR07}, which is basis-independent. Transforming by $\matB^{-1}$, we see that
\[ \matB^{-1} \vecu \pmod{\Z^n} \approx U([0,1)^n),\]
where $\approx$ denotes small total variation distance, and $U(\cdot)$ denotes the uniform distribution. Thus, we have converted worst-case structure into average-case structure (that is independent of $\matB$).

We can repeat this process $m$ times as follows: sample $\matU \sim \Norm(\zero, \sigma^2 \matI_n)^m$, where now
\[ \matB^{-1} \matU \pmod{\Z^{n \times m}} \approx U([0,1)^{n \times m}).\]
We then set $\matA = \matB^{-1} \matU \pmod{\Z^{n \times m}}$ and feed $\matA$ into the $\SBP$ solver. (For simplicity, in this overview, we assume that $\SBP$ allows input matrices that are uniform over $[0,1)^{n \times m}$ instead of Gaussian, but in the proof, we resolve this distinction by sampling from an appropriate discrete Gaussian distribution.) We get back some $\vecx \in \{\pm 1\}^m$ so that $\norm{\matA \vecx}_{\infty} \leq \sbpError(m/n) \sqrt{m}$, or in other words, 
\[ \matA \vecx - \vece = \zero, \]
where $\norm{\vece}_{\infty} \leq \sbpError(m/n) \sqrt{m}$. Since $\vecx \in \Z^m$, this implies
\[ \matB^{-1} \matU \vecx - \vece = \zero \mod \Z^{n \times m}. \]
Multiplying by $\matB$ on the left gives
\[ \matU \vecx - \matB \vece = \zero \mod \LL(\matB),\]
or in other words, $\matU \vecx - \matB \vece \in \LL(\matB)$. To see that $\matU \vecx - \matB \vece$ is slightly smaller than the vectors in $\matB$, note that $\norm{\matU \vecx}_2$ can be upper-bounded by a basis-independent quantity, since $\sigma$ was chosen basis-independently. The bottleneck term is $\norm{\matB \vece}_2$, which is smaller than columns of $\matB$ as long as $\norm{\vece}_2$ is sufficiently small. If $\sbpError(x) = 1/x^{1/2 + \varepsilon}$ for $\varepsilon > 0$, then
\[ \norm{\vece}_{\infty} \leq \sbpError(m/n) \sqrt{m} = \frac{n^{1/2 + \varepsilon}}{m^{\varepsilon}}, \]
so as long as we set $m$ large enough so that $m^\varepsilon \gg n^{1/2 + \varepsilon}$, we can force $\norm{\matB \vece}_2$ to be small enough to produce a smaller lattice vector than anything in $\matB$, as desired.

\paragraph{Allowing $\vecx \in \{-1, 0, 1\}^m$ instead of $\vecx \in \{\pm 1\}^m$.} One could define a variant of $\SBP$ where we just need $\vecx \in \{-1, 0, 1\}^m \setminus \{\zero\}$ instead of $\vecx \in \{\pm 1\}^m$. This is an easier problem, as including zero entries in $\vecx$ is a simple way to decrease $\norm{\matA \vecx}_{\infty}$. However, our reduction idea above actually \emph{still} works in this setting. For more details, see \Cref{sec:variants-and-generalizations}.

\subsubsection{Reduction to \texorpdfstring{$\NPP$}{NPP}}
The reduction to $\NPP$ is quite similar to the reduction to $\SBP$, but with one additional trick that converts between vectors and scalars. This trick has been used in \cite{RegevLectureNotes,DBLP:conf/tcc/BrakerskiV15} for a similar reason, and we follow their footsteps.
Explicitly, it is the \emph{Chinese remainder theorem}: for distinct primes $p_1, \dots, p_n$ and $q = \prod_{i \in [n]} p_i$, there is a group isomorphism
\[ \varphi : \bigoplus_{i \in [n]} \Z/ p_i \Z \longrightarrow \Z / q\Z. \]
We will scale this isomorphism so that
\[ \widetilde{\varphi} : \bigoplus_{i \in [n]} 1/p_i \cdot \Z/ p_i \Z \longrightarrow 1/q \cdot \Z / q\Z \]
is (a) invariant to integer shifts in the input and (b) $\Z$-linear, in the sense that $\widetilde{\varphi}(\vecx) = \vecc^\top \vecx \pmod{1}$ for some $\vecc \in \Z^n$.

Following the $\SBP$ reduction above, we can set $\matA = \matB^{-1} \matU \pmod{\Z^{n \times m}}$ and then appropriately ``round'' $\matA$ to get uniformly random $\matA' = \floor{\matA}_{\vecp} \in \left(\bigoplus_{i \in [n]} 1/p_i \cdot \Z/ p_i \Z \right)^m$. We then apply $\widetilde{\varphi}$ column-wise to $\matA'$ to get $\veca' \in (1/q \cdot \Z/q\Z)^m \subseteq [0,1)^m$. By adding small uniform noise to $\veca'$, we can get some $\veca \sim [0,1)^m$. We feed this into our $\NPP$ solver to get some $\vecx \in \{\pm 1\}^m$ such that $|\veca^\top \vecx| \leq \nppError(m) \sqrt{m}$. By using $\Z$-linearity of $\widetilde{\varphi}$, and assuming $\nppError$ is sufficiently small so that there are no ``wraparound'' errors in $\widetilde{\varphi}^{-1}$, we can recover a somewhat smaller lattice vector $\vecs \in \LL(\matB)$ than we started with, just like in the $\SBP$ reduction.

\subsubsection{Alternate Reduction Paths}\label{sec:alternate-reductions}
We briefly mention other paths to reduce worst-case lattice problems to $\SBP$ and $\NPP$, using existing intermediate problems and known worst-case to average-case reductions.

\paragraph{SBP.} Instead of starting from worst-case lattice problems, we could start with a noisy version of the \emph{short integer solutions} (SIS) problem in the $\ell_\infty$ norm, defined roughly as follows: Given random $\matA \sim (\Z/q\Z)^{n \times m}$, output $\vecx \in \Z^m \setminus \{\zero\}$ such that $\norm{\vecx}_{\infty}$ and $\norm{\matA \vecx \pmod{q} }_{\infty}$ are small.\footnote{One can remove the $\pmod{q}$ constraint to reduce from SIS over $\Z$, and this would work as well.} The original worst-case to average-case reductions for lattice problems indeed reduce worst-case lattice problems to (average-case) SIS \cite{DBLP:conf/stoc/Ajtai96,DBLP:journals/siamcomp/MicciancioR07}. One can reduce (noisy) SIS to $\SBP$ by adding small noise $\matE \sim U\left([0, 1/q)^{n \times m} \right)$ and setting
\[ \matA' = \frac{1}{q} \matA + \matE \sim U\left([0, 1)^{n \times m} \right).\]
(For simplicity, we assume here that $\SBP$ allows matrices with i.i.d. $U([0,1))$ entries instead of standard normal, but we can remove this assumption by sampling from the appropriate discrete Gaussian and scaling.) Feeding $\matA'$ into the $\SBP$ solver yields $\vecx \in \{\pm 1\}^m$ such that $\norm{\matA \vecx}_{\infty}$ is small, assuming $q\norm{\matE \vecx}_{\infty}$ is also sufficiently small. 

Another approach for $\SBP$ is to reduce directly from Continuous Learning With Errors (CLWE)~\cite{DBLP:conf/stoc/Bruna0ST21}. This average-case problem, which is known to be as hard as worst-case approximate lattice problems~\cite{DBLP:conf/stoc/Bruna0ST21,DBLP:conf/focs/GupteVV22}, asks to distinguish $(\matA, \vecs^\top \matA + \vece^\top \pmod{1})$ from $(\matA, \vecb^\top)$, where $\matA \sim \Norm(0,1)^{n \times m}$, $\vecs$ is drawn spherically from sufficiently large sphere, $\vece$ is drawn from a Gaussian with small variance, and $\vecb$ is a uniformly random vector mod $1$. By using integrality of $\vecx$ and simultaneous smallness of $\vecx$ and $\matA \vecx$, one can simply multiply the second argument on the right by $\vecx$ to distinguish from the uniform distribution modulo $1$. However, we find this CLWE approach somewhat unsatisfactory since it would show hardness of $\SBP$ assuming at least one of (a) \emph{quantum} worst-case hardness of \emph{search} lattice problems or (b) \emph{classical} worst-case hardness of \emph{decisional} lattice problems~\cite{DBLP:conf/stoc/Bruna0ST21,DBLP:conf/focs/GupteVV22}. On the other hand, our main result shows hardness of $\SBP$ under only \emph{classical} worst-case hardness of \emph{search} lattice problems. As search lattice problems are known to be harder than decisional ones, our main approach in this work yields a stronger result than the CLWE approach. Indeed, for those familiar with lattice-based cryptography, this should not come as a surprise: the difference between $\SBP$ and CLWE is reminiscent of the difference between SIS and the Learning With Errors (LWE) problems in lattice-based cryptography.

\paragraph{NPP.} Brakerski and Vaikuntanathan~\cite{DBLP:conf/tcc/BrakerskiV15} define a problem called the \emph{one-dimensional short integer solutions} problem (\textrm{1D}-\textrm{SIS}), defined roughly as follows: Given $\veca \sim (\Z / q \Z)^m$, output $\vecx \in \Z^m \setminus \{\zero\}$ such that  $\norm{\vecx}_{\infty}$ and $|\veca^\top \vecx \pmod{q}|$ are small, where $q$ is a product of $n$ primes.\footnote{Similarly here, one can remove the $\pmod{q}$ constraint and reduce from \textrm{1D}-\textrm{SIS} over $\Z$.} \cite{DBLP:conf/tcc/BrakerskiV15} shows a worst-case to average-case reduction from worst-case lattice problems in dimension $n$ to \textrm{1D}-\textrm{SIS}. One can reduce \textrm{1D}-\textrm{SIS} to $\NPP$ by adding small noise $\vece \sim U\left( [0, 1/q)^m\right)$ and setting
\[ \veca' = \frac{1}{q} \veca + \vece \sim U\left( \left[0, 1 \right)^m \right).\]
(For simplicity, we similarly assume here that $\NPP$ works for vectors with i.i.d. $U([0,1))$ entries.) Feeding $\veca'$ into the $\NPP$ solver yields $\vecx \in \{\pm 1\}^m$ such that $|\veca^\top \vecx|$ is small, assuming $q|\vece^\top \vecx|$ is also sufficiently small.
\\\\While these approaches would work, we view this introduction of parameter $q \in \Z$ or modulus reduction as extraneous and misleading. To show hardness of a discrete optimization problem in continuous Euclidean space ($\SBP$ and $\NPP$), we should ideally start from a problem that is itself a discrete optimization problem in continuous Euclidean space (worst-case lattice problems). There is no need to add premature discretization by introducing the parameter $q$ or unnecessary modular reduction.

\section{Preliminaries}
For a predicate $\varphi$, we use the notation $\one[\varphi] \in \{0,1\}$ to denote the indicator variable of whether $\varphi$ is true ($1$) or false ($0$). We use $\ln(\cdot)$ to denote the natural logarithm (base $e$) and $\log(\cdot)$ to denote $\log_2(\cdot)$. We use $\R_{>0}$ to refer to the set of all positive real numbers, and we use $\R_{\geq 1}$ to refer to the set of all real numbers that are at least $1$. We say a function $f : \N \to \R_{>0}$ is \emph{negligible} if for all $c \in \N$, 
\[ \lim_{n \to \infty} n^c \cdot f(n) = 0. \]
We say a function $f : \N \to \R$ is \emph{non-negligible} if there exists some $c \in \N$ such that $f(n) \geq 1/n^c$ for all sufficiently large $n$.

For two distributions $\mathcal{D}_1$, $\mathcal{D}_2$, we use the notation $\Delta(\mathcal{D}_1, \mathcal{D}_2)$ to denote the total variation distance between $\mathcal{D}_1$ and $\mathcal{D}_2$, which we refer to simply as the \emph{statistical distance} between the two distributions. We say that two distributions are \emph{statistically close} if their statistical distance is negligible (in some implicit parameter, typically $n$ or $m$ for us).

For a set $S$, we let $U(S)$ denote the uniform distribution over $S$. (If $S$ is not finite and $S \subseteq \R^n$ is Lebesgue measurable, this will be uniform with respect to the standard Lebesgue measure.)

\subsection{Norms \& Matrices}
We use the notation $\zero$ to denote the all $0$s vector in $\R^n$, where the dimension $n$ is clear from context. We use $\matI_n \in \R^{n \times n}$ to denote the $n$-dimensional identity matrix. We use the standard $\ell_1, \ell_2, \ell_{\infty}$ norms on $\R^n$. For a matrix $\matA \in \R^{n \times m}$, we use the notation $\sigma_{\max}(\matA)$ to denote the spectral norm, or maximum singular value, of $\matA$. Explicitly,
\[ \sigma_{\max}(\matA) = \max_{\vecv \neq \zero} \frac{\norm{\matA \vecv}_2}{ \norm{\vecv}_2}. \]
For a matrix $\matA \in \R^{n \times m}$, we often write $\matA$ by its columns, as in $\matA = [\veca_1, \veca_2, \cdots, \veca_m]$ for $\veca_i \in \R^n$. We sometimes abuse notation and move interchangeably between matrices $\matA \in \R^{n \times m}$ and tuples of $m$ vectors in $\R^n$ (as defined by the columns of $\matA$). For a matrix $\matA \in \R^{n \times m}$ given by $\matA = [\veca_1, \veca_2, \cdots, \veca_m]$, we use the notation
\[ \norm{\matA} = \max_{j \in [m]} \norm{\veca_j}_2. \]
We emphasize that $\norm{\matA}$ does \emph{not} refer to the standard spectral norm on matrices.
\begin{lemma}\label{bound-by-matrix-norm-and-ell-1}
    For $\matA \in \R^{n \times m}$ and $\vecv \in \R^m$, we have the inequality $\norm{\matA \vecv}_2 \leq \norm{\matA} \norm{\vecv}_1$.
\end{lemma}
\begin{proof}
    Let $\matA = [\veca_1, \cdots, \veca_m]$, and let $\vecv$ have entries $v_j \in \R$. We have
    \[\norm{\matA \vecv}_2 = \norm{\sum_{j = 1}^m v_j \veca_j}_2 \leq \sum_{j = 1}^m \norm{ v_j \veca_j}_2 = \sum_{j = 1}^m |v_j| \norm{\veca_j}_2 \leq \sum_{j = 1}^m |v_j| \norm{\matA} = \norm{\matA} \norm{\vecv}_1.\]
\end{proof}

We also use the following basic fact.

\begin{lemma}\label{ell-1-to-ell-infty}
    For all $\vecv \in \R^n$, $\norm{\vecv}_1  \leq n  \norm{\vecv}_{\infty}$.
\end{lemma}

\subsection{Lattices}\label{latticeprelims}

For an invertible matrix $\matB \in \R^{n \times n}$, an $n$-dimensional lattice generated by basis $\matB$, denoted $\LL(\matB)$, is given by all integer linear combinations of columns of $\matB$. That is,
\[ \LL(\matB) := \{ \matB \vecz : \vecz \in \Z^n\}. \]
We define the (half-open) parallelepiped $\PP(\matB)$ to be the set
\[ \PP(\matB) := \{ \matB \vecv : \vecv \in [0, 1)^n\}.\]
Note that for all $\vecx \in \R^n$, there exists unique $\vecy \in \PP(\matB)$ such that $\vecx - \vecy \in \LL(\matB)$. We use the notation $\vecy = \vecx \pmod \matB$ to denote the corresponding $\vecy \in \PP(\matB)$ for a given $\vecx \in \R^n$. Note that $\vecy$ is computable in polynomial time given $\matB$ and $\vecx$. For a lattice $\Lambda$, we denote the \emph{dual lattice} of $\Lambda$ as $\Lambda^*$, defined by
\[ \Lambda^* = \{ \vecx \in \R^n : \forall \vecv \in \Lambda, \inner{\vecx, \vecv} \in \Z\}.\]

For $i \in [n]$, we can define the $i$th successive minimum of a lattice $\Lambda$ to be the smallest $\lambda_i$ such that there exist $i$ linearly independent lattice points of $\ell_2$ norm at most $\lambda_i$. Letting $B$ denote the unit ball, this can be phrased as
\[ \lambda_i(\Lambda) := \min\{r : \dim(\Span(\Lambda \cap r B)) \geq i \}.\]
Note that $\lambda_1(\Lambda)$ is the \emph{minimum distance} of $\Lambda$.

We also define the \emph{covering radius} $\nu(\Lambda)$ of a lattice $\Lambda$, defined by
\[ \nu(\Lambda) = \max_{\vecx \in \R^n} \min_{\vecv \in \Lambda} \norm{\vecx - \vecv}_2. \]
That is, $\nu(\Lambda)$ is the smallest real number such that every element of $\R^n$ has distance at most $\nu(\Lambda)$ from (some point in) $\Lambda$.

We now define some fundamental (worst-case) lattice problems.

\begin{definition}[Shortest Independent Vectors Problem]
    For a parameter $\gamma : \N \to \R_{\geq 1}$, the \emph{shortest independent vectors problem ($\SIVP_{\gamma}$)} is a (worst-case) search problem defined as follows. Given an invertible basis $\matB \in \R^{n \times n}$ as input, output $n$ vectors $\matS = [\vecs_1, \cdots, \vecs_n] \in \R^{n \times n}$ such that the following hold:
    \begin{itemize}
        \item $\matS$ is linearly independent.
        \item For all $i \in [n]$, $\vecs_i \in \LL(\matB)$;
        \item $\norm{\matS} \leq \gamma(n) \cdot \lambda_n(\LL(\matB))$.
    \end{itemize}
\end{definition}

\begin{definition}[Covering Radius Problem]
    For a parameter $\gamma : \N \to \R_{\geq 1}$, the \emph{gap covering radius problem ($\GapCRP_{\gamma}$)} is a (worst-case) decision problem defined as follows. Given an invertible basis $\matB \in \R^{n \times n}$ and threshold $\theta \in \R_{>0}$ as input, output $1$ if $\nu(\LL(\matB)) \leq \theta$ and $0$ if $\nu(\LL(\matB)) > \gamma(n) \cdot \theta$.
\end{definition}

\begin{definition}[Guaranteed Distance Decoding]
    For a parameter $\gamma : \N \to \R_{\geq 1}$, the \emph{guaranteed distance decoding ($\GDD_{\gamma}$)} is a (worst-case) search problem defined as follows. Given an invertible basis $\matB \in \R^{n \times n}$ and target vector $\vect \in \R^n$ as input, output a lattice point $\vecx \in \LL(\matB)$ such that $\norm{\vect - \vecx}_2 \leq \gamma(n) \cdot \lambda_n(\LL(\matB))$.
\end{definition}
We note that $\GDD$ is often defined using $\nu(\cdot)$ instead of $\lambda_n(\cdot)$, as using $\nu(\cdot)$ guarantees solutions for all $\gamma \geq 1$. However, throughout this paper, we will be in the regime where $\gamma(n) \geq \sqrt{n}/2$, which guarantees a solution even with $\lambda_n(\cdot)$ because $\lambda_n(\Lambda) \geq \frac{2}{\sqrt{n}} \nu(\Lambda)$ for all full-rank lattices $\Lambda$ \cite[Lemma 4.3]{DBLP:conf/coco/GuruswamiMR04}.

We will also use an intermediate problem called Incremental Guaranteed Distance Decoding ($\IncGDD$) defined by \cite{DBLP:journals/siamcomp/MicciancioR07}.
\begin{definition}[Incremental GDD, Definition 5.6 of \cite{DBLP:journals/siamcomp/MicciancioR07}]\label{def:incgdd}
    For a parameter $\gamma : \N \to \R_{\geq 1}$, the \emph{incremental guaranteed distance decoding ($\IncGDD_{\gamma}$)} problem is a (worst-case) search problem defined as follows. Given as input:
    \begin{itemize}
        \item An invertible basis $\matB \in \R^{n \times n}$,
        \item A set $\matS$ of $n$ linearly independent vectors in $\LL(\matB)$ (represented as columns of $\matS \in \R^{n \times n}$),
        \item A target point $\vect \in \R^n$, and
        \item A parameter $r > \gamma(n) \cdot \lambda_n(\LL(\matB))$,
    \end{itemize}
    output some vector $\vecs \in \LL(\matB)$ such that $\norm{\vecs - \vect}_2 \leq r + \dfrac{\norm{\matS}}{8}$.
\end{definition}

\paragraph{Comparison to \cite{DBLP:journals/siamcomp/MicciancioR07}.} \Cref{def:incgdd} differs from \cite[Definition 5.6]{DBLP:journals/siamcomp/MicciancioR07} in two ways. First, instead of using $\lambda_n(\LL(\matB))$, \cite{DBLP:journals/siamcomp/MicciancioR07} considers general functions $\phi$ mapping $n$-dimensional lattices to $\R_{>0}$. Second, instead of fixing the constant $8$, \cite{DBLP:journals/siamcomp/MicciancioR07} parametrizes this more generally by some constant $g$.
\\\\As shown in~\cite{DBLP:journals/siamcomp/MicciancioR07}, there are reductions from these worst-case lattice problems to $\IncGDD$.
\begin{lemma}[Lemma 5.10 in \cite{DBLP:journals/siamcomp/MicciancioR07}]\label{sivp-to-incgdd}
    For any $\gamma(n) \geq 1$, there is a reduction from $\SIVP_{8\gamma}$ to $\IncGDD_{\gamma}$.
\end{lemma}

\begin{lemma}[Combining Lemmas 5.11, 5.12 in \cite{DBLP:journals/siamcomp/MicciancioR07}]\label{gapcrp-to-incgdd}
    For any $\gamma(n) \geq 1$, there is a (randomized) reduction from $\GapCRP_{12\gamma}$ to $\IncGDD_{\gamma}$.
\end{lemma}

\begin{lemma}[Lemma 5.11 in \cite{DBLP:journals/siamcomp/MicciancioR07}]\label{gdd-to-incgdd}
    For any $\gamma(n) \geq 1$, there is a reduction from $\GDD_{3\gamma}$ to $\IncGDD_{\gamma}$.
\end{lemma}

The best known algorithms for these lattice problems employ \emph{lattice reduction} techniques, based on LLL and BKZ \cite{lenstra1982factoring,DBLP:journals/tcs/Schnorr87}. They currently all have the following time-approximation trade-off: For a tunable parameter $k$, in dimension $n$, it is possible to solve lattice problems with approximation factor $\gamma = 2^{\widetilde{O}(n/k)}$ in time $2^{\widetilde{O}(k)}$, where $\widetilde{O}(\cdot)$ hides polylog factors in $n$. Equalizing these terms gives a $2^{\widetilde{O}(\sqrt{n})}$-time algorithm to solve $2^{\widetilde{O}(\sqrt{n})}$-approximate lattice problems. Despite extensive study in the lattice literature, no better algorithms are known (that improve the exponent by more than a polylog factor). In particular, the following two assumptions stand:

\begin{assumption}[Polynomial Hardness of Approximate Worst-case Lattice Problems]\label{polynomial-hardness-assumption}
For every polynomial $\gamma(n)$, at least one of $\SIVP_{\gamma}$, $\GapCRP_{\gamma}$, or $\GDD_{\gamma}$ requires super-polynomial time to solve.
\end{assumption}

\begin{assumption}[Subexponential Hardness of Approximate Worst-case Lattice Problems]\label{subexp-hardness-assumption}
For all constants $\varepsilon > 0$, at least one of $\SIVP_{\gamma}$, $\GapCRP_{\gamma}$, or $\GDD_{\gamma}$ requires time $2^{\omega(n^{1/2 - \varepsilon})}$ to solve, where $\gamma(n) = 2^{n^{1/2 - \varepsilon}}$.
\end{assumption}

\subsection{Continuous and Discrete Gaussian Measures}
For $\mu \in \R$ and $\sigma \in \R_{>0}$, we use the notation $\Norm(\mu, \sigma^2)$ to denote the (univariate) Normal distribution with mean $\mu$ and standard deviation $\sigma$. More generally, in $n \in \N$ variables, for $\vecmu \in \R^n$ and positive semi-definite $\matSigma \in \R^{n \times n}$, we use the notation $\Norm(\vecmu, \matSigma)$ to denote the multivariate Gaussian distribution with mean $\vecmu$ and covariance matrix $\matSigma$. For the special case where $\matSigma = \sigma^2 \matI_n$ for some $\sigma \in \R_{>0}$, let $\varphi_{\sigma, \vecmu}(\cdot)$ denote the probability density function of $\Norm(\vecmu, \sigma^2 \matI_n)$, i.e.,
\[ \varphi_{\sigma, \vecmu}(\vecx) = \frac{1}{(2\pi \sigma^2)^{n/2}} \cdot \exp \left( - \frac{ \norm{\vecx - \vecmu}^2_2}{2\sigma^2} \right).\]

We now recall standard Gaussian measure notions from the lattice literature (specifically, \cite{DBLP:journals/siamcomp/MicciancioR07}). For an input $\vecx \in \R^n$, we define the Gaussian function $\rho_{s, \vecmu}(\cdot)$ centered at $\vecmu \in \R^n$ with scale parameter $s \in \R_{>0}$ to be
\[ \rho_{s, \vecmu}(\vecx) = \exp \left( - \frac{ \pi \norm{\vecx - \vecmu}^2_2}{s^2} \right) = s^n \cdot \varphi_{s/\sqrt{2\pi}, \vecmu}(\vecx). \]
Therefore, $\rho_{s, \vecmu}(\vecx)/s^n$ is the density function for the probability distribution $\Norm(\vecmu, \frac{s^2}{2\pi} \matI_n)$.

For any discrete $S \subseteq \R^n$, let $\rho_{s, \vecmu}(S)$ denote the sum
\[ \sum_{\vecv \in S} \rho_{s, \vecmu}(\vecv) \in \R. \]
Let $\Lambda$ be a full-rank $n$-dimensional lattice. We define the \emph{discrete Gaussian distribution} $D_{\Lambda + \vecmu, s}$ shifted by $\vecmu \in \R^n$ and scale parameter $s \in \R_{>0}$ to be the distribution with probability mass function
\begin{align*} D_{\Lambda + \vecmu, s}(\vecx) &= \begin{cases}
    \dfrac{\rho_{s, \zero}(\vecx)}{\rho_{s, \zero}(\Lambda + \vecmu)} = \dfrac{\rho_{s, \zero}(\vecx)}{\rho_{s, -\vecmu}(\Lambda)} & \text{if $\vecx \in \Lambda$},
    \\0 & \text{otherwise},
\end{cases}
\\&= \one[\vecx \in \Lambda] \cdot \frac{\exp \left( - \pi \norm{\vecx - \vecmu}^2_2 /s^2 \right)}{\sum_{\vecv \in \Lambda} \exp \left( - \pi \norm{\vecx - \vecmu}^2_2 /s^2 \right)}.
\end{align*}
We will use the fact that there is an efficient sampler for the discrete Gaussian for the special case where $\Lambda = \Z^n$ (see, e.g., \cite[Lemma 2.3]{DBLP:conf/stoc/BrakerskiLPRS13}).

For an $n$-dimensional lattice $\Lambda$ and $\epsilon > 0$, we define the \emph{smoothing parameter} $\eta_{\epsilon}(\Lambda)$ to be the smallest $s \in \R_{>0}$ such that $\rho_{1/s, \zero}(\Lambda^* \setminus \{\zero\}) \leq \epsilon$. We now recall standard results about the smoothing parameter.

\begin{lemma}[Lemma 3.3 of \cite{DBLP:journals/siamcomp/MicciancioR07}]\label{smoothing-bound-lambda-n}
    For any $n$-dimensional lattice $\Lambda$ and $\epsilon > 0$,
    \[ \eta_{\epsilon}(\Lambda) \leq \sqrt{\frac{\ln(2n(1 + 1/\epsilon))}{\pi}} \cdot \lambda_n(\Lambda). \]
\end{lemma}

\begin{lemma}[Lemma 4.1 of \cite{DBLP:journals/siamcomp/MicciancioR07}]\label{smoothing-mod-P}
    For any lattice $\LL(\matB)$ and any $\epsilon > 0$, $s \geq \eta_{\epsilon}(\LL(\matB))$, and $\vecmu \in \R^n$, we have the statistical distance bound
    \[ \Delta\left(\Norm\left(\vecmu, \frac{s^2}{2\pi} \matI_n \right) \mod \PP(\matB), U(\PP(\matB)) \right) \leq \frac{\epsilon}{2}. \]
\end{lemma}

\begin{corollary}\label{above-lambda-n-is-uniform}
    For any full-rank $\matB \in \R^{n \times n}$, $\epsilon > 0$, and $\vecmu \in \R^n$, if
    \[ \sigma \geq \sqrt{\frac{\ln(2n(1 + 1/\epsilon))}{2\pi^2}} \cdot \lambda_n(\LL(\matB)), \] then we have the statistical distance bound
    \[ \Delta\left(\Norm\left(\vecmu, \sigma^2 \matI_n \right) \mod \PP(\matB), U(\PP(\matB)) \right) \leq \frac{\epsilon}{2}. \]
\end{corollary}
\begin{proof}
    This follows by combining \Cref{smoothing-bound-lambda-n,smoothing-mod-P}, where we set $\sigma = s / \sqrt{2\pi}$.
\end{proof}

\begin{lemma}\label{uniform-shift-discrete-gaussian-is-gaussian}
    For $\sigma > 0$, let $\mathcal{D}_{\sigma}$ denote the distribution of outputs when first sampling $\vecv \sim U([0,1)^n)$ and then outputting a sample from $D_{\Z^n + \vecv, \sigma \sqrt{2\pi}}$. For any $\epsilon \in (0, 1/2)$, if
    \[ \sigma \geq \sqrt{\frac{\ln(2n(1 + 1/\epsilon))}{2 \pi^2}}, \]
    then we have the statistical distance bound
    \[ \Delta\left(\Norm\left(\zero, \sigma^2 \matI_n \right), \mathcal{D}_{\sigma} \right) \leq 4 \epsilon.\]
\end{lemma}
\begin{proof}[Proof sketch of \Cref{uniform-shift-discrete-gaussian-is-gaussian}]
    Let $s = \sigma \sqrt{2\pi}$. The analysis in \cite[Lemma 17]{DBLP:conf/focs/GupteVV22} shows that the distributions have statistical distance at most
    \[ \sup_{\vecv \in [0,1)^n} \frac{\rho_{s, \zero}(\Z^n)}{\rho_{s, \zero}(\Z^n + \vecv)} - 1 = \sup_{\vecv \in [0,1)^n} \frac{\rho_{s, \zero}(\Z^n)}{\rho_{s, -\vecv}(\Z^n)} - 1. \]
    Implicit in \cite[Lemma 4.4]{DBLP:journals/siamcomp/MicciancioR07} is that as long as $s \geq \eta_{\epsilon}(\Z^n)$, then
    \[ \frac{\rho_{s, \zero}(\Z^n)}{\rho_{s, -\vecv}(\Z^n)} \in \left[1, \frac{1 + \epsilon}{1 - \epsilon} \right] \subseteq [1, 1 + 4 \epsilon], \]
    where the last inclusion comes from the bound $\epsilon < 1/2$. Subtracting by $1$ and invoking \Cref{smoothing-bound-lambda-n} with $\Lambda = \Z^n$ gives the desired bound.
\end{proof}

We now recall some standard spectral bounds for Gaussian matrices.

\begin{lemma}[As in \cite{rudelson2010non}]\label{gaussian-singular-value}
Let $\matA \in \R^{n \times m}$ be such that $A_{i,j} \sim_{\iid} \Norm(0,1)$. For all $t > 0$, we have
\[ \Pr\left[\sigma_{\max}(\matA) \leq \sqrt{n} + \sqrt{m} + t \right] \geq 1 - 2 e^{-t^2/2}.\]
In particular, for $m \geq 16n$, by setting $t = \sqrt{n}$, we have
\[\Pr\left[\sigma_{\max}(\matA) \leq \frac{3\sqrt{m}}{2} \right] \geq 1 - 2 e^{-n/2}.\]
\end{lemma}

We now recall standard tail bounds for the $\chi^2(n)$ distribution, corresponding to $\ell_2$-norm bounds on Gaussian vectors.
\begin{lemma}[As in \cite{laurent2000adaptive}, Corollary of Lemma~1]\label{chi-squared-tail-bound}
For any $t \geq 0$, we have
\[ \Pr_{\vecv \sim \Norm(\zero, \matI_n)}\left[\norm{\vecv}_2^2 \geq n + 2 \sqrt{tn} + 2t \right] \leq e^{-t}. \]
In particular, setting $t = n/4$, we get
\[ \Pr_{\vecv \sim \Norm(\zero, \matI_n)}\left[\norm{\vecv}_2 \geq \sqrt{5/2} \cdot \sqrt{n} \right] \leq e^{-n/4}.\]
\end{lemma}

\subsection{Symmetric Perceptrons and Number Partitioning}

\begin{definition}[Symmetric Binary Perceptron Problem]
    For a parameter $\sbpError : \R_{\geq 1} \to [0,1]$, the \emph{symmetric binary perceptron ($\SBP_{\sbpError}$)} problem is an average-case search problem defined as follows. Given a random Gaussian matrix $\matA \sim \Norm(0,1)^{n \times m}$ as input where $m \geq n$, output a vector $\vecx \in \{-1,1\}^m$ such that $\norm{\matA \vecx}_{\infty} \leq \sbpError(m/n) \cdot \sqrt{m}$.
\end{definition}

\begin{definition}[Number Partitioning Problem]
    For a parameter $\nppError : \N \to [0,1]$, the \emph{number partitioning problem ($\NPP_{\nppError}$)} is an average-case search problem defined as follows. Given a random Gaussian vector $\veca \sim \Norm(\zero,\matI_m)$, output a vector $\vecx \in \{-1,1\}^m$ such that $|\veca^\top \vecx| \leq \nppError(m) \cdot \sqrt{m}$.
\end{definition}
We emphasize that $\NPP_{\nppError}$ is exactly a special case of $\SBP_{\sbpError}$ (when setting $n = 1$).

\subsection{Worst-case to Average-case Reductions}

We recall basic notions of reductions between (search) worst-case and average-case computational problems. Let $A \in \mathsf{FNP}$ be a worst-case search problem, and let $B \in \mathsf{FNP}$ be an average-case search problem defined over some distribution family $\{\mathcal{D}_n\}_{n \in \N}$. We say that there is a $T(n)$-time reduction from $A$ to $B$ if for all non-negligible functions $\mu$, there exists a randomized $\poly(T)$-time oracle Turing machine $M^{(\cdot)}$ such that for all (possibly randomized) $\oracle$ such that
\[ \Pr_{y \gets \mathcal{D}_n} [(y, \oracle(y)) \in B] \geq \mu(n), \]
 for all $n \in \N$, it holds that for all $x \in \{0,1\}^*$,
\[ \Pr_{r}\left[ \left(x, M^{\oracle}(x; r) \right) \in A \right] \geq \frac{2}{3}, \]
where $r$ denotes the internal randomness of $M$. We emphasize that the polynomial in the $\poly(T)$ run-time of $M^{(\cdot)}$ may depend on the non-negligible function $\mu$. Since $x$ is worst-case and $A \in \mathsf{FNP}$, standard amplification applies to make the success probability of $M$ exponentially close to $1$.

\section{Reduction to Symmetric Binary Perceptrons}
\label{sec:sbpmain}

\begin{theorem}\label{main-technical-reduction-to-sbp}
Suppose $\sbpError(x) = 1/x^{1/2 + \varepsilon}$ for some constant $\varepsilon > 0$. There exists some polynomial $\gamma(n) = n^{O(1/\varepsilon)}$ such that there is a polynomial-time reduction from $\IncGDD_{\gamma}$ to $\SBP_{\sbpError}$.
\end{theorem}

We now state our main corollary for symmetric binary perceptrons.

\begin{corollary}\label{sivp-to-sbp}
    Suppose there is a polynomial time algorithm for $\SBP_{\sbpError}$ (on average) for $\sbpError(x) = 1/x^{1/2 + \varepsilon}$ for some constant $\varepsilon > 0$ that succeeds with non-negligible probability. Then, there are randomized polynomial time algorithms for the (worst-case) lattice problems $\SIVP_{\gamma}$, $\GapCRP_{\gamma}$, and $\GDD_{\gamma}$ for some polynomial $\gamma(n)$. In particular, Assumption \ref{polynomial-hardness-assumption} implies that there is no polynomial time algorithm for $\SBP_{\sbpError}$ (on average) for $\sbpError(x) = 1/x^{1/2 + \varepsilon}$ for some constant $\varepsilon > 0$ that succeeds with non-negligible probability.

    More generally, if there is a $T(n,m)$-time algorithm for $\SBP_{\sbpError}$ (on average) for $\sbpError(x) = 1/x^{1/2 + \varepsilon}$ for some constant $\varepsilon > 0$ that succeeds with non-negligible probability, then there are randomized $\poly(n, T(n, \poly(n)))$-time algorithms for the (worst-case) lattice problems $\SIVP_{\gamma}$, $\GapCRP_{\gamma}$, and $\GDD_{\gamma}$ for some polynomial $\gamma(n)$.
\end{corollary}
\begin{proof}[Proof of \Cref{sivp-to-sbp}]
    This follows by directly composing \Cref{sivp-to-incgdd,gapcrp-to-incgdd,gdd-to-incgdd} and \Cref{main-technical-reduction-to-sbp}.
\end{proof}

\begin{remark}\label{sharper-sbp-reduction}
    For a sharper bound on $\sbpError$ in \Cref{main-technical-reduction-to-sbp} and \Cref{sivp-to-sbp}, we can instead assume subexponential hardness of approximate worst-case lattice problems. Specifically, for $\sbpError(x) = \frac{1}{\sqrt{x} \log^{1 + c}(x)}$ where $c > 0$, we can set $m = 2^{O\left(n^{3/(2c)} \right)}$ and $\gamma(n) = 2^{O\left(n^{3/(2c)} \right)}$ in the reduction from $\IncGDD$. In particular, for $c > 3$, Assumption~\ref{subexp-hardness-assumption} implies that there is no polynomial time algorithm for $\SBP_{\sbpError}$ (on average) for $\sbpError(x) = \frac{1}{\sqrt{x} \log^{1 + c}(x)}$ that succeeds with non-negligible probability. 
\end{remark}

\subsection{Proof of Theorem~\ref{main-technical-reduction-to-sbp}}

We now prove \Cref{main-technical-reduction-to-sbp}.

\begin{proof}[Proof of \Cref{main-technical-reduction-to-sbp}]
Let $(\matB \in \R^{n \times n}, \matS \in \R^{n \times n}, \vect \in \R^n, r \in \R)$ be the given $\IncGDD$ instance. Let
\begin{align*} \sigma_2 &= \ln n,
\\m &= \ceil{\left(8 \sigma_2 n^{3/2 + \varepsilon} \right)^{1/\varepsilon}} = n^{ \Theta(1/\varepsilon)},
\\\sigma_1 &= \frac{r}{4m},
\\\gamma(n) &= 4 m \ln n = n^{\Theta(1/\varepsilon)}.
\end{align*}
Sample $\vecu_1 \sim \Norm(\vect, \sigma_1^2 \matI_n)$, $\vecu_2, \cdots, \vecu_{m} \sim_{\iid} \Norm(\zero, \sigma_1^2 \matI_n)$. Let $\matU = [\vecu_1, \vecu_2, \cdots, \vecu_m] \in \R^{n \times m}$. Sample $m$ uniformly random lattice vectors $\vecv_i \in \LL(\matB) \mod \PP(\matS)$ (see~\cite[Proposition 2.9]{DBLP:journals/siamcomp/Micciancio04}), and let $\matV = [\vecv_1, \vecv_2, \cdots, \vecv_m] \in \R^{n \times m}$. Define
\[ \widetilde{\matA} = \matS^{-1}(\matV + \matU) \mod{\Z^n} \in [0,1)^{n \times m}.\]
\begin{proposition}\label{tilde-A-is-uniform}
The distribution of $\widetilde{\matA}$ is statistically close to $U([0,1)^{n \times m})$.
\end{proposition}
\begin{proof}[Proof of \Cref{tilde-A-is-uniform}]
    Recall that by definition of $\IncGDD$, we know
    \[ r > \gamma(n) \cdot \lambda_n(\LL(\matB)) = 4m \ln n \cdot \lambda_n(\LL(\matB)).\]
    In anticipation of applying \Cref{above-lambda-n-is-uniform}, we observe that for sufficiently large $n$,
    \[ \sigma_1 = \frac{r}{4m} \geq \ln n \cdot \lambda_n(\LL(\matB)) \geq \sqrt{\frac{\ln\left(2n \left(1 + 1/e^{-\ln^2(n)} \right)\right)}{2 \pi^2}} \cdot \lambda_n(\LL(\matB)). \]
    Therefore, we can invoke \Cref{above-lambda-n-is-uniform} $m$ times (once with $\vecmu = \vect$, $m-1$ times with $\vecmu = \zero$) and the triangle inequality to see that 
    \[ \Delta\left(\matU \mod{\PP(\matB)}, U(\PP(\matB))^m \right) \leq m \cdot \frac{e^{-\ln^2(n)}}{2} = \negl(n).\]
    Since $\matV$ has columns that are uniform elements of $\LL(\matB) \mod{\PP(\matS)}$, and since $\LL(\matS) \subseteq \LL(\matB)$, it follows that 
    \[ \Delta\left(\matU + \matV  \mod{\PP(\matS)}, U(\PP(\matS))^m \right) \leq \negl(n).\]
    Multiplying on the left by $\matS^{-1}$ gives
    \[ \Delta\left(\widetilde{\matA}, U(\PP(\Z^n))^m\right) \leq \negl(n), \]
    or equivalently, that $\widetilde{\matA}$ is statistically close to uniform over $[0,1)^{n \times m}$.
\end{proof}
For $j \in [m]$, let $\widetilde{\veca}_j \in [0,1)^n$ denote the $j$th column of $\widetilde{\matA}$. For all $j \in [m]$, sample $\vecw_j \sim D_{\Z^n + \widetilde{\veca}_{j}, \sigma_2 \sqrt{2\pi}}$. Let $\matW = [\vecw_1, \cdots, \vecw_m]$. Note that by construction, $\vecw_j \equiv \widetilde{\veca}_j \mod \Z^n$, i.e., $\matW \equiv \widetilde{\matA} \mod \Z^n$. In anticipation of applying \Cref{uniform-shift-discrete-gaussian-is-gaussian}, we observe that for sufficiently large $n$,
\[ \sigma_2 = \ln(n) \geq \sqrt{\frac{\ln\left(2n\left(1 + 1/e^{-\ln^2(n)} \right)\right)}{2 \pi^2}}.\] Therefore, by invoking \Cref{uniform-shift-discrete-gaussian-is-gaussian} $m$ times, the triangle inequality, and \Cref{tilde-A-is-uniform}, we see that
\[ \Delta\left(\matW, \Norm\left(0,\sigma_2^2\right)^{n \times m}\right) \leq 4m e^{-\ln^2(n)} + \negl(n) = \negl(n). \]
Let $\matA = \frac{1}{\sigma_2}  \matW$. It follows that $\matA$ is statistically close to $\Norm(0, 1)^{n \times m}$.

Feed $\matA$ into the $\SBP_{\sbpError}$ solver to receive some $\vecx \in \{-1, 1\}^m$ such that $\norm{\matA \vecx}_{\infty} \leq \sbpError(m/n) \cdot \sqrt{m}$. Let $\vece = - \matA \vecx \in \R^n$ with $\norm{\vece}_{\infty} \leq \sbpError(m/n) \cdot \sqrt{m}$. Let $x_1 \in \{-1, 1\}$ be the first entry of $\vecx$. Let $\vece' = \sigma_2 \vece \in \R^n$. The reduction outputs $\vecs = x_1(\matU \vecx + \matS \vece') \in \R^n$.

We first argue that $\vecs \in \LL(\matB)$. Since $\matA \vecx + \vece = \zero$, by scaling up by $\sigma_2$, we have $\matW \vecx + \vece' = \zero$. Since $\vecx \in \Z^m$,
\[ \zero = \matW \vecx + \vece' \equiv \widetilde{\matA} \vecx + \vece' \equiv  \matS^{-1}(\matV + \matU) \vecx + \vece' \mod \Z^n, \]
meaning that $\matS^{-1} (\matV + \matU) \vecx + \vece' \in \Z^n$, and thus
\[ (\matV + \matU) \vecx + \matS \vece' \in \LL(\matS) \subseteq \LL(\matB). \]
Since $\matV$ contains vectors in $\LL(\matB)$ and $\vecx \in \Z^m$, we can subtract by $\matV \vecx$ to get
\[ \matU \vecx + \matS \vece' \in \LL(\matB).\]
Multiplying by the sign $x_1 \in \{-1, 1\}$ gives
\[ \vecs = x_1(\matU \vecx + \matS \vece') \in \LL(\matB), \]
as desired.

We next argue that the norm of $\vecs - \vect$ is small. Decompose $\vecx$ as $\vecx^\top = [x_1 || \vecx_{-1}^\top]$, and decompose $\matU$ as $\matU = [\vecu_1 || \matU_{-1}]$. We then have
\begin{align*}
\vecs = x_1(\matU \vecx + \matS \vece') &= x_1 (\vecu_1 x_1 + \matU_{-1} \vecx_{-1} + \matS \vece')
\\&= x_1^2 \vecu_1 + x_1\matU_{-1} \vecx_{-1} + x_1 \matS \vece'
\\&= \vecu_1 + x_1\matU_{-1} \vecx_{-1} + x_1 \matS \vece'
\\&= \vect + \vecu_1' + x_1\matU_{-1} \vecx_{-1} + x_1 \matS \vece',
\end{align*}
where the distribution of $\vecu_1'$ is $\Norm(\zero, \sigma_1^2 \matI_n)$. It follows that with all but negligible probability,
\begin{align*}
\norm{\vecs - \vect}_2 &= \norm{\vecu_1' + x_1\matU_{-1} \vecx_{-1} + x_1\matS \vece'}_2 &
\\&\leq \norm{\vecu_1'}_2 + \norm{\matU_{-1} \vecx_{-1}}_2 + \norm{\matS \vece'}_2 &
\\&\leq \sigma_1 \sqrt{\frac{5n}{2}} + \sigma_{\max}(\matU_{-1}) \norm{\vecx_{-1}}_2 + \norm{\matS \vece'}_2 & \text{(by \Cref{chi-squared-tail-bound})}
\\&\leq \sigma_1 \sqrt{\frac{5n}{2}} + \sigma_{\max}(\matU_{-1}) \sqrt{m-1} + \norm{\matS} \norm{\vece'}_1 & \text{(by \Cref{bound-by-matrix-norm-and-ell-1})}
\\&\leq \sigma_1 \sqrt{\frac{5n}{2}}  + \sigma_{\max}(\matU_{-1}) \sqrt{m-1} + n \norm{\matS} \norm{\vece'}_{\infty} & \text{(by \Cref{ell-1-to-ell-infty})}
\\&\leq \sigma_1 \sqrt{\frac{5n}{2}}  + \frac{3\sigma_1 \sqrt{m}}{2} \cdot \sqrt{m-1} +  n \norm{\matS} \norm{\vece'}_{\infty} & \text{(by \Cref{gaussian-singular-value})}
\\&\leq 4\sigma_1 m + n \norm{\matS} \norm{\vece'}_{\infty} & 
\\&\leq r + n \norm{\matS} \norm{\vece'}_{\infty}. & 
\end{align*}
Therefore, it suffices to show that $\norm{\vece'}_{\infty} \leq 1/(8n)$. Recall that we have
\begin{align*}
\norm{\vece'}_{\infty} = \sigma_2 \norm{\vece}_{\infty} \leq \sigma_2 \cdot \sbpError(m/n) \cdot \sqrt{m} = \sigma_2 \cdot \left(\frac{n}{m}\right)^{1/2 + \varepsilon} \sqrt{m} = \sigma_2 \cdot \frac{n^{1/2 + \varepsilon}}{m^{\varepsilon}}.
\end{align*}
Since $m \geq (8 \sigma_2 n^{3/2 + \varepsilon})^{1/\varepsilon} $, we have $\norm{\vece}_{\infty} \leq 1/(8n)$, as desired.

Lastly, we note that the $\SBP_{\sbpError}$ solver need only succeed with some non-negligible probability $\mu$. As a result, we can repeat this whole process $O(1/\mu) = \poly(n, m)$ times, and since we can efficiently verify whether the $\SBP_{\sbpError}$ solver succeeded, the reduction will still go through.
\end{proof}

\subsection{Variants and Generalizations}\label{sec:variants-and-generalizations}

We mention a few variants of $\SBP$ for which the reduction in \Cref{main-technical-reduction-to-sbp} would also apply. As they are not critical to our main result, for simplicity, we only sketch the justifications.
\begin{enumerate}
    \item \textbf{Uniform $\matA$}. Instead of having $\matA \sim \Norm(0, 1)^{n \times m}$, if we had a $\SBP$ solver that worked with $\matA \sim U([0,1])^{n \times m}$, our reduction would actually be simpler and more direct. In particular, there would be no need for discrete Gaussian sampling.
    
    \item\label{item:zero-entries-in-x} \textbf{Zero entries in $\vecx$}. Instead of requiring $\vecx \in \{\pm 1\}^m$ from the $\SBP$ solver, if we instead allowed $\vecx \in \{-1, 0, 1\}^m$ with $\vecx \neq \zero$ from the $\SBP$ solver, a similar reduction to the one in \Cref{main-technical-reduction-to-sbp} would work as well. The main difference is that the reduction would instead ``guess'' a coordinate $j \in [m]$ for which $x_j \neq 0$ (with success probability at least $1/m$) and put the vector $\vect$ in the mean of that coordinate, instead of the first coordinate. (See \cite{DBLP:journals/siamcomp/MicciancioR07} for more rigorous details.)

    A priori, the version of $\SBP$ that allows zero-entries in $\vecx$ is in fact a lot easier. In particular, for $\sbpError(x) = 1/\sqrt{x}$, setting $\vecx = (1, 0, \dots, 0)^\top$ would get $\norm{\matA \vecx}_{\infty} \leq \widetilde{O}(1) \ll \sbpError(m/n) \sqrt{m} = \sqrt{n}$ with high probability by standard Gaussian tail bounds. However, in our reduction, we have $\sbpError(x) = 1/x^{1/2 + \varepsilon}$. The bound on $\norm{\matA \vecx}_{\infty}$ is
    \[\norm{\matA \vecx}_{\infty} \leq \sbpError(m/n) \sqrt{m} = \frac{\sqrt{m}}{(m/n)^{1/2 + \varepsilon}} = \frac{n^{1/2 + \varepsilon}}{m^\varepsilon}. \]
    In our reduction, we set $m$ so that $m^\varepsilon \gg n^{1/2 + \varepsilon}$, making $\norm{\matA \vecx}_{\infty} \ll 1$, in particular, a stronger requirement than $\norm{\matA \vecx}_{\infty} \leq \widetilde{O}(1)$.
    
    \item \textbf{Larger $\vecx \in \Z^m$}. Instead of (in particular) requiring $\norm{\vecx}_{\infty} \leq 1$ from the $\SBP$ solver, one could relax this requirement to $\norm{\vecx}_{\infty} \leq B$ for some larger bound $B \in \N$. In addition to the modifications discussed in \Cref{item:zero-entries-in-x}, we would put the vector $\vect/z$ in the mean of that coordinate, where $z \sim U(\{-B, -B+1, \dots, B-1, B\} \setminus \{0\})$. The runtime, $\gamma$, and $\sbpError$ would now worsen by a factor of $\Theta(B)$. (See \cite{DBLP:journals/siamcomp/MicciancioR07} for more rigorous details.)
\end{enumerate}

\section{Reduction to Number Partitioning}\label{sec:nbp}

\begin{lemma}[Chinese Remainder Theorem]\label{plain-crt}
    Let $p_1, \dots, p_n \in \N$ be distinct positive prime numbers. For $q = \prod_{i \in [n]} p_i$, there is a group isomorphism
    \[ \varphi : \bigoplus_{i \in [n]} \Z / p_i \Z \longrightarrow \Z/q\Z. \]
    Moreover, this map can be written as
    \[ \varphi: (y_1, \cdots, y_n) \mapsto \sum_{i \in [n]} c_i y_i \]
    for $c_i \in \Z$ such that $q/p_i$ divides $c_i$ for all $i \in [n]$ (so that this map is well-defined). Furthermore, the inverse map $\varphi^{-1} : \Z/q\Z \to \bigoplus_{i \in [n]} \Z / p_i \Z$ can be written as
    \[ \varphi^{-1} : z \mapsto (z, z, \dots, z),\]
    where for all $i \in [n]$, $z \in \Z/q\Z = \{0, \dots, q-1\}$ is interpreted directly as an element of $\Z/p_i\Z = \{0, \dots, p_i-1\}$ by reduction modulo $p_i$.
\end{lemma}

\begin{lemma}[Normalized CRT]\label{normalized-crt}
    Let $p_1, \dots, p_n \in \N$ be distinct positive prime numbers. For $q = \prod_{i \in [n]} p_i$, there is a group isomorphism
    \[ \widetilde{\varphi} : \bigoplus_{i \in [n]} 1/p_i \cdot  \Z / p_i \Z \longrightarrow 1/q \cdot \Z/q\Z. \]
    Moreover, there exists an integer vector $\vecc \in \Z^n$ such that this map can be written as
    \[ \widetilde{\varphi} : (y_1, \cdots, y_n) \mapsto \vecc^\top \vecy = \sum_{i \in [n]} c_i y_i. \]
    Furthermore, the inverse map $\widetilde{\varphi}^{-1} : 1/q \cdot \Z/q\Z \to \bigoplus_{i \in [n]} 1/p_i \cdot \Z / p_i \Z$ can be written as
    \[ \widetilde{\varphi}^{-1} : z \mapsto \left(\frac{q}{p_1} z, \frac{q}{p_2} z, \dots, \frac{q}{p_n} z \right).\]
\end{lemma}
\begin{proof}[Proof of \Cref{normalized-crt}]
    This follows directly by \Cref{plain-crt}, by setting
    \[ \widetilde{\varphi}(y_1, \dots, y_n) = 1/q \cdot \varphi(p_1 y_1, p_2 y_2, \dots, p_n y_n).  \]
\end{proof}

Letting $\vecp = (p_1, \dots, p_n) \in \Z^n$, we use the notation $\floor{\cdot}_{\vecp} : [0,1)^n \to \bigoplus_{i \in [n]} 1/p_i \cdot \Z/p_i \Z$ to denote the function
\[ \floor{\cdot}_{\vecp} : \vecv \mapsto \left( \frac{\floor{v_1 p_1}}{p_1} , \dots, \frac{\floor{v_1 p_n}}{p_n} \right). \]
More generally, for elements in $[0,1)^{n \times m}$, we extend $\floor{\cdot}_{\vecp} : [0,1)^{n \times m} \to \left( \bigoplus_{i \in [n]} 1/p_i \cdot \Z/p_i \Z \right)^m$ to operate column-wise.

We will use the following basic fact.

\begin{lemma}\label{rounding-error-bound}
For any $\matA \in [0,1)^{n \times m}$ and $\vecx \in \R^m$, 
\[ \norm{ \left( \matA - \floor{\matA}_{\vecp} \right) \vecx}_1 \leq \frac{n}{\min_{i \in [n]} p_i} \norm{\vecx}_1. \]
\end{lemma}
\begin{proof}
    Let $p^* = \min_{i \in [n]} p_i$. Letting $\matM = \matA - \floor{\matA}_{\vecp}$, by properties of the floor function, we have $\matM \in \left[0,  \frac{1}{p^*} \right)^{n \times m}$. Let $\vecm_1, \cdots, \vecm_n \in \left[0, \frac{1}{p^*}\right)^m$ be the rows of $\matM$. We have
    \[ \norm{\matM \vecx}_1 = \sum_{i \in [n]} \left| \vecm_i^\top \vecx \right| \leq \sum_{i \in [n]} \norm{\vecm_i}_{\infty} \norm{\vecx}_1 \leq \frac{n}{p^*} \norm{\vecx}_1, \]
    as desired, where we have used H{\"o}lder's inequality to see that $\left|\vecm_i^\top \vecx \right| \leq \norm{\vecm_i}_{\infty} \norm{\vecx}_1$.
\end{proof}
We also use the following fact about the density of prime numbers.
\begin{lemma}\label{prime-density}
    For all sufficiently large $N \in \N$, there exist at least $N/\ln(N)$ distinct prime numbers in the interval $[N, 10N]$. 
\end{lemma}
\begin{proof}
    For $N \in \N$, let $\pi(N) = |\{a \in [N] : a \text{ is prime} \}|$ denote the prime-counting function. By the prime number theorem, we know that for all sufficiently large $N$,
    \[ \frac{N}{2 \ln N} \leq \pi(N) \leq \frac{2N}{\ln N}. \]
    In particular,
    \[ \pi(N) \leq \frac{2N}{\ln N}, \;\;\; \pi(10N) \geq \frac{10N}{2 \ln(10 N)} = \frac{10N}{2\ln(N) + 2 \ln(10)} > \frac{4N}{ \ln N},\]
    for sufficiently large $N$. Therefore, for sufficiently large $N$, by taking the difference of the two quantities, there are at least $N / \ln N$ primes in $[N, 10N]$.
\end{proof}

\begin{theorem}\label{main-technical-reduction-to-npp}
Suppose $\nppError(m) = 2^{-\log^{2 + \varepsilon} m}$ for some constant $\varepsilon > 0$. Then there exists $\gamma(n) = 2^{O\left(n^{\frac{1}{1 + \varepsilon}}\right)}$ such that there is a $\poly(m)$-time reduction from $\IncGDD_{\gamma}$ in dimension $n = \Omega((\log m)^{1 + \varepsilon})$ to $\NPP_{\nppError}$ (in dimension $m$).
\end{theorem}

We now state our main corollary for number partitioning.

\begin{corollary}\label{sivp-to-npp}
    Suppose there is a polynomial time algorithm for $\NPP_{\nppError}$ (on average) for $\nppError(m) = 2^{-\log^{2 + \varepsilon} m}$ for some constant $\varepsilon > 0$ that succeeds with non-negligible probability. Then, there are randomized $2^{O\left( n^{\frac{1}{1 + \varepsilon}} \right)}$-time algorithms for the (worst-case) lattice problems $\SIVP_{\gamma}$, $\GapCRP_{\gamma}$, and $\GDD_{\gamma}$ in dimension $n$ for $\gamma(n) = 2^{O\left(n^{\frac{1}{1 + \varepsilon}}\right)}$. In particular, Assumption~\ref{subexp-hardness-assumption} implies that for all constant $\varepsilon > 0$, there is no polynomial time algorithm for $\NPP_{\nppError}$ (on average) for $\nppError(m) = 2^{-\log^{3 + \varepsilon} m}$ that succeeds with non-negligible probability. 

    More generally, suppose there is a $T(m)$-time algorithm for $\NPP_{\nppError}$ (on average) for $\nppError(m) = 2^{-\log^{2 + \varepsilon} m}$ for some constant $\varepsilon > 0$ that succeeds with non-negligible probability. Then, there are randomized $T\left( 2^{O\left( n^{\frac{1}{1 + \varepsilon}} \right)} \right)$-time algorithms for the (worst-case) lattice problems $\SIVP_{\gamma}$, $\GapCRP_{\gamma}$, and $\GDD_{\gamma}$ in dimension $n$ for $\gamma(n) = 2^{O\left(n^{\frac{1}{1 + \varepsilon}}\right)}$.
\end{corollary}

\begin{proof}[Proof of \Cref{sivp-to-npp}]
    This follows by directly composing \Cref{sivp-to-incgdd,gapcrp-to-incgdd,gdd-to-incgdd} and \Cref{main-technical-reduction-to-npp}.
\end{proof}

We now prove \Cref{main-technical-reduction-to-npp}.

\begin{proof}[Proof of \Cref{main-technical-reduction-to-npp}]
Let $(\matB \in \R^{n \times n}, \matS \in \R^{n \times n}, \vect \in \R^n, r \in \R)$ be the given $\IncGDD$ instance. Let
\begin{align*} 
m &= 2^{10n^{\frac{1}{1 + \varepsilon}}},
\\\sigma_1 &= \frac{r}{4m},
\\\sigma_2 &= \ln m,
\\\gamma(n) &= 4 m \ln m = 40 \ln(2)  2^{10n^{\frac{1}{1 + \varepsilon}}} n^{\frac{1}{1 + \varepsilon}}.
\end{align*}
Let $p_1, \dots, p_n$ be $n$ distinct prime numbers in the range $[32 n m, 320 n m]$, which we know must exist for sufficiently large $n, m$ by \Cref{prime-density} (since $m \geq n$). Let $q = \prod_{i = 1}^n p_i \leq (320nm)^n$. Sample $\vecu_1 \sim \Norm(\vect, \sigma_1^2 \matI_n)$, $\vecu_2, \cdots, \vecu_{m} \sim_{\iid} \Norm(\zero, \sigma_1^2 \matI_n)$. Let $\matU = [\vecu_1, \vecu_2, \cdots, \vecu_m] \in \R^{n \times m}$. Sample $m$ uniformly random lattice vectors $\vecv_i \in \LL(\matB) \mod \PP(\matS)$ (see~\cite[Proposition 2.9]{DBLP:journals/siamcomp/Micciancio04}), and let $\matV = [\vecv_1, \vecv_2, \cdots, \vecv_m] \in \R^{n \times m}$. Define
\[ \matA = \matS^{-1}(\matV + \matU) \mod{\Z^n} \in [0,1)^{n \times m}.\]
\begin{proposition}\label{npp-A-is-uniform}
We have the inequality
\[ \Delta \left(\matA,  U([0,1)^{n \times m}) \right) \leq \negl(m). \]
\end{proposition}
\begin{proof}
    This follows from the same analysis as in \Cref{tilde-A-is-uniform}. Specifically, since for sufficiently large $n$ and $m$,
    \[ \sigma_1 = \frac{r}{4m} \geq \ln m \cdot \lambda_n(\LL(\matB)) \geq \sqrt{\frac{\ln\left(2n \left(1 + 1/e^{-\ln^2(m)} \right)\right)}{2 \pi^2}} \cdot \lambda_n(\LL(\matB)),\]
    by \Cref{above-lambda-n-is-uniform} and the triangle inequality, we get a total statistical distance of $m \cdot e^{-\ln^2(m)} = \negl(m)$.
\end{proof}

Let $\widetilde{\varphi}$ be the isomorphism guaranteed by \Cref{normalized-crt}, with $\vecc \in \Z^n$ being the coefficients defining the linear map $\widetilde{\varphi}$. Let $\vecy \in \R^m$ be defined by
\[ \vecy = \widetilde{\varphi}(\floor{\matA}_{\vecp}) + \vecf \mod \Z^m \in [0,1)^m\]
where $\vecf \in \R^m$ is sampled as $\vecf \sim U([0,1/q)^m)$.

We argue that $\Delta(\vecy, U([0,1)^m)) \leq \negl(m)$ as follows. Since $\matA$ is (close to) $U([0,1)^{n \times m})$, it follows that $\floor{\matA}_{\vecp}$ is (close to) $U(\bigoplus_{i \in [n]} 1/p_i \cdot  \Z / p_i \Z)^m$. Moreover, since $\widetilde{\varphi}$ is a bijection, it follows that $\widetilde{\varphi}(\floor{\matA}_{\vecp})$ is (close to) $U(1/q \cdot \Z/q\Z)^m$. Since $\vecf \sim U([0,1/q)^m)$, we can see that $\Delta(\vecy, U([0,1)^m)) \leq \negl(m)$.

Sample $\vecw \sim D_{\Z^m + \vecy, \sigma_2 \sqrt{2\pi}}$. Note that by construction, $\vecw \equiv \vecy \mod \Z^m$. In anticipation of applying \Cref{uniform-shift-discrete-gaussian-is-gaussian}, we observe that for sufficiently large $m$,
\[ \sigma_2 = \ln m \geq \sqrt{\frac{\ln\left(2m\left(1 + 1/e^{-\ln^2(m)} \right)\right)}{2 \pi^2}}.\] 
Therefore, by invoking \Cref{uniform-shift-discrete-gaussian-is-gaussian}, $\Delta(\vecy, U([0,1)^m)) \leq \negl(m)$, and the triangle inequality, we see that
\[ \Delta \left( \Norm(\zero, \sigma_2^2 \matI_m), \vecw \right) \leq e^{-\ln^2(m)} + \negl(m) = \negl(m). \]
Let $\veca = \frac{1}{\sigma_2} \vecw$. It follows that $\veca$ is statistically close to $\Norm(\zero, \matI_m)$. Feed $\veca$ into the $\NPP$ solver to receive some $\vecx \in \{-1, 1\}^m$ such that $|\veca^\top \vecx| \leq \nppError(m) \sqrt{m}$. Let $e = - \veca^\top \vecx \in \R$ with $|e| \leq \nppError(m) \sqrt{m}$. Let $x_1 \in \{-1, 1\}$ be the first entry of $\vecx$. Let $e' = \sigma_2 e \in \R$, and let $e'' = \vecf^\top \vecx + e' \in \R$.

The reduction outputs
\[ \vecs = x_1 \left( \matU \vecx - \matS( \matA - \floor{\matA}_{\vecp}) \vecx + \matS \widetilde{\varphi}^{-1}(e'')\right). \]

We first argue that $\vecs \in \LL(\matB)$. Since $\veca^\top \vecx + e = 0$, by scaling up, we have $\vecw^\top \vecx + e' = 0$. Since $\vecx \in \Z^m$,
\begin{align*} 0 = \vecw^\top \vecx + e' \equiv \vecy^\top \vecx + e' &=  \left( \widetilde{\varphi}(\floor{\matA}_{\vecp}) + \vecf \right)^\top \vecx + e' \mod 1,
\\&\equiv  \vecc^\top \floor{\matA}_{\vecp} \vecx + \vecf^\top \vecx + e' \mod 1,
\\&\equiv \widetilde{\varphi} (\floor{\matA}_{\vecp} \vecx) + e'' \mod 1.
\end{align*}
By closure, we know $e'' \in 1/q \cdot \Z/q\Z$, so we have
\[ 0 \equiv \widetilde{\varphi} (\floor{\matA}_{\vecp} \vecx) + e'' \equiv \widetilde{\varphi} (\floor{\matA}_{\vecp} \vecx) + \widetilde{\varphi}\left(\widetilde{\varphi}^{-1}(e'') \right) \equiv \widetilde{\varphi} \left( \floor{\matA}_{\vecp} \vecx + \widetilde{\varphi}^{-1}(e'') \right) \mod{1}. \]
By applying $\widetilde{\varphi}^{-1}$, it follows that
\[ \matA \vecx - (\matA - \floor{\matA}_{\vecp}) \vecx + \widetilde{\varphi}^{-1}(e'') = \floor{\matA}_{\vecp} \vecx + \widetilde{\varphi}^{-1}(e'') \in \Z^n. \]
Plugging in the definition of $\matA$ and since $\vecx \in \Z^m$,
\[  \matS^{-1}(\matV + \matU) \vecx - (\matA - \floor{\matA}_{\vecp}) \vecx + \widetilde{\varphi}^{-1}(e'') \in \Z^n. \]
Multiplying by $\matS$ on the left gives
\[ (\matV + \matU) \vecx - \matS( \matA - \floor{\matA}_{\vecp}) \vecx + \matS \widetilde{\varphi}^{-1}(e'') \in \LL(\matS) \subseteq \LL(\matB).\]
Since $\matV$ contains vectors in $\LL(\matB)$ and $\vecx \in \Z^m$, we can subtract by $\matV \vecx$ to get
\[ \matU \vecx - \matS( \matA - \floor{\matA}_{\vecp}) \vecx + \matS \widetilde{\varphi}^{-1}(e'') \in \LL(\matB).\]
Multiplying by the sign $x_1 \in \{-1, 1\}$ gives
\[ \vecs = x_1 \left( \matU \vecx - \matS( \matA - \floor{\matA}_{\vecp}) \vecx + \matS \widetilde{\varphi}^{-1}(e'')\right) \in \LL(\matB), \]
as desired.

We next argue that the norm of $\vecs - \vect$ is small. Decompose $\vecx$ as $\vecx^\top = [x_1 || \vecx_{-1}^\top]$, and decompose $\matU$ as $\matU = [\vecu_1 || \matU_{-1}]$. We then have
\begin{align*} \vecs &= x_1 \matU \vecx - x_1 \matS( \matA - \floor{\matA}_{\vecp}) \vecx + x_1 \matS \widetilde{\varphi}^{-1}(e'')
\\&= x_1 (\vecu_1 x_1 + \matU_{-1} \vecx_{-1}) - x_1 \matS( \matA - \floor{\matA}_{\vecp}) \vecx + x_1 \matS \widetilde{\varphi}^{-1}(e'')
\\&= \vecu_1 + x_1 \matU_{-1} \vecx_{-1} - x_1 \matS( \matA - \floor{\matA}_{\vecp}) \vecx + x_1 \matS \widetilde{\varphi}^{-1}(e'') 
\\&= \vect + \vecu_1' + x_1 \matU_{-1} \vecx_{-1} - x_1 \matS( \matA - \floor{\matA}_{\vecp}) \vecx + x_1 \matS \widetilde{\varphi}^{-1}(e''),
\end{align*}
where the distribution of $\vecu_1'$ is $\Norm(\zero, \sigma_1^2 \matI_n)$. It follows that
\begin{align*} \norm{\vecs - \vect}_2 &= \norm{\vecu_1' + x_1 \matU_{-1} \vecx_{-1} - x_1 \matS( \matA - \floor{\matA}_{\vecp}) \vecx + x_1 \matS \widetilde{\varphi}^{-1}(e'')}_2 &
\\&\leq\norm{\vecu_1'}_2 + \norm{\matU_{-1} \vecx_{-1}}_2 + \norm{\matS( \matA - \floor{\matA}_{\vecp}) \vecx}_2 + \norm{\matS \widetilde{\varphi}^{-1}(e'')}_2 &
\\&\leq 4 \sigma_1 m + \norm{\matS( \matA - \floor{\matA}_{\vecp}) \vecx}_2 + \norm{\matS \widetilde{\varphi}^{-1}(e'')}_2 & \text{(by \Cref{chi-squared-tail-bound,gaussian-singular-value})}
\\&\leq 4 \sigma_1 m  + \norm{\matS} \left( \norm{ \left(\matA - \floor{\matA}_{\vecp} \right) \vecx}_1 + \norm{\widetilde{\varphi}^{-1}(e'')}_1 \right) & \text{(by \Cref{bound-by-matrix-norm-and-ell-1})}
\\&\leq 4 \sigma_1 m  + \norm{\matS} \left( \frac{n}{\min_{i \in [n]} p_i} \norm{\vecx}_1 + \norm{\widetilde{\varphi}^{-1}(e'')}_{1} \right) & \text{(by \Cref{rounding-error-bound})}
\\&= r  + \norm{\matS} \left( \frac{nm}{\min_{i \in [n]} p_i} + \norm{\widetilde{\varphi}^{-1}(e'')}_{1} \right). &
\\&\leq r  + \norm{\matS} \left( \frac{1}{16} + \norm{\widetilde{\varphi}^{-1}(e'')}_{1} \right). &
\end{align*}
It suffices to upper bound $\norm{\widetilde{\varphi}^{-1}(e'')}_{1}$ by $1/16$. Recall from \Cref{normalized-crt} that
\[ \widetilde{\varphi}^{-1}(e'') = \left(\frac{q}{p_1} e'', \frac{q}{p_2} e'', \dots, \frac{q}{p_n} e'' \right) \in \bigoplus_{i \in [n]} 1/p_i \cdot \Z/p_i \Z.\]
For of these entries to be small when viewed in $\R$, we want to ensure that there's no ``wraparound.'' In particular, if
\[ e'' \leq \frac{\min_{i \in [n]} p_i}{16qn}, \]
then $\norm{\widetilde{\varphi}^{-1}(e'')}_{1}\leq 1/16$, as desired. Since $\min_{i \in [n]} p_i \geq 32 nm$, it suffices to show that
\[ e'' \leq \frac{2m}{q}. \]
Recall that $e'' = \vecf^\top \vecx + e' = \vecf^\top \vecx + \sigma_2 e$, where $\vecf \sim U([0,1/q)^m)$ and $|e| \leq \nppError(m) \sqrt{m}$. It follows that
\[ |e''| \leq \left| \vecf^\top \vecx \right| + |e'| \leq  \norm{\vecf}_{\infty} \norm{\vecx}_1 + \sigma_2 |e| \leq \frac{m}{q} + \sigma_2 \cdot \nppError(m) \sqrt{m}.\]
Thus, it suffices to show $\sigma_2 \cdot \nppError(m) \sqrt{m} \leq m/q$, or equivalently, $\nppError(m) \leq \sqrt{m} / (q \ln m)$. We have
\begin{align*} \frac{\sqrt{m}}{q \ln m} = \frac{2^{5 n^{\frac{1}{1 + \varepsilon}}}}{ q 10 \ln(2) n^{\frac{1}{1 + \varepsilon}}} &\geq \frac{2^{5 n^{\frac{1}{1 + \varepsilon}}}}{ 10 \ln(2) (320n m)^n n^{\frac{1}{1 + \varepsilon}}}
\\&= \frac{2^{5 n^{\frac{1}{1 + \varepsilon}}}}{ 10 \ln(2) (320n)^n  2^{10 n^{\frac{2 + \varepsilon}{1 + \varepsilon}}} n^{\frac{1}{1 + \varepsilon}}}
\\&\geq \frac{1}{2^{11 n^{\frac{2 + \varepsilon}{1 + \varepsilon}}}}
\end{align*}
for sufficiently large $n$. On the other hand,
\[ \nppError(m) = \frac{1}{2^{(\log m)^{2 + \varepsilon}}} = \frac{1}{2^{\left(10 n^{\frac{1}{1 + \varepsilon}}\right)^{2 + \varepsilon}}} = \frac{1}{2^{10^{2 + \varepsilon} n^{\frac{2 + \varepsilon}{1 + \varepsilon}}}} \leq  \frac{1}{2^{100 n^{\frac{2 + \varepsilon}{1 + \varepsilon}}}}. \]
Therefore, $\nppError(m) \leq \sqrt{m} / (q \ln m)$, as desired.

Lastly, we note that the $\NPP_{\nppError}$ solver need only succeed with some non-negligible probability $\mu(m)$. As a result, we can repeat this whole process $O(1/\mu) = \poly(m)$ times, and since we can efficiently verify whether the $\NPP_{\nppError}$ solver succeeded, the reduction will still go through.
\end{proof}

Moreover, as $\NPP$ is a special case of $\SBP$, all of the generalizations and variants discussed in \Cref{sec:variants-and-generalizations} apply here to $\NPP$ as well.

\paragraph{Acknowledgements.} 
The authors were supported in part by DARPA under Agreement No. HR00112020023, NSF CNS-2154149 and a Simons Investigator Award. The first author is also supported in part by NSF DGE-2141064. We are particularly thankful to David Gamarnik for a stimulating conversation where he described the symmetric perceptron and number partitioning problems to us and posed the question of showing computational hardness for them. We are grateful to Alon Rosen, Kiril Bangachev, Stefan Tiegel and Riccardo Zecchina for valuable discussions. We also thank the anonymous reviewers for their valuable feedback.

\bibliographystyle{alpha}
\bibliography{refs}

\end{document}